\documentclass[10pt,a4paper]{article}
\usepackage{amsmath,amssymb,amsthm,esint,bm}%,mathrsfs
\usepackage{mathrsfs}
\usepackage{bookmark}
\allowdisplaybreaks[3]

\allowdisplaybreaks

\usepackage[T1]{fontenc}
\usepackage[utf8]{inputenc}
\usepackage{authblk}
\usepackage{amsmath,amssymb,amsthm,esint,bm}%,mathrsfs
\usepackage{mathrsfs}
\usepackage{bookmark}
\usepackage{amsmath}

\newtheorem{definition}{Definition}[section]
\newtheorem{theorem}[definition]{Theorem}
\newtheorem{lemma}[definition]{Lemma}

\newtheorem{corollary}[definition]{Corollary}
\theoremstyle{remark}
\newtheorem{remark}[definition]{Remark}
\numberwithin{equation}{section}

\newcommand{\ddiv}{\operatorname{div}}
\newcommand{\ls}{\leqslant}
\newcommand{\rs}{\geqslant}
\newcommand{\non}{\nonumber}

\title{Gradient  potential estimates  for elliptic  double obstacle problems with Orlicz growth}

\author[a]{Qi Xiong}
\author[b]{Zhenqiu Zhang\thanks{Corresponding author.}}
\author[c]{Lingwei Ma}

\affil[a]{School of Mathematics, Southwest Jiaotong University, Chengdu, Sichuan, 610031, P.R. China}
\affil[b]{School of Mathematical Sciences and LPMC, Nankai University, Tianjin, 300071, P.R. China}
\affil[c]{School of Mathematical Sciences, Tianjin Normal University, Tianjin, 300387, P.R. China}
\date{\today}
\usepackage{hyperref}
\begin{document}
\maketitle
\footnotetext[1]{E-mail: xq@swjtu.edu.cn(Q. Xiong),  zqzhang@nankai.edu.cn (Z. Zhang),mlw1103@163.com (L. Ma).}

\maketitle
\begin{abstract}
In this paper,we consider the solutions of the elliptic double obstacle problems  with Orlicz growth involving  measure data.  Some pointwise estimates for the approximable solutions to these problems  are obtained in terms of  fractional maximal operators. Furthermore, we establish pointwise and  oscillation estimates for the gradients of solutions  via the non-linear Wolff potentials, which in turn  yield $C^{1,\alpha}$-regularity of solutions.\\


\end{abstract}
%\renewcommand{\thepage}{\roman{page}}
%\setcounter{page}{1}

%%\tableofcontents

\section{Introduction and main results}\label{section1}

We consider  the double obstacle problems with Orlicz growth , which are closely connected to measure data problems of the type
\begin{equation}\label{1.1}
  -\operatorname{div}\left({a}(x,Du)\right) =\mu \quad\quad\mbox{in}\ \ \ \Omega, \\[0.05cm]
\end{equation}
where $ \Omega\subseteq \mathbb{R}^n, n\geqslant2 $ is a bounded open set and $\mu \in \mathcal{M}_{b}(\Omega)$, where $\mathcal{M}_{b}(\Omega)$  is the set of signed Radon measures $\mu$ for which $|\mu|(\Omega)$ is finite and  here we  denote by $|\mu|$ the total variation of $\mu$. Moreover we
assume that $\mu(\mathbb{R}^n \backslash \Omega)=0$ and $a=a(x,\eta): \Omega \times\mathbb{R}^n \rightarrow \mathbb{R}^n$ is measurable for each $x\in \Omega$ and differentiable for almost every $\eta \in \mathbb{R}^n$ and there exist constants $0<l\leqslant 1 \leqslant L<+\infty$ such that for all $x \in \Omega,\eta,\lambda \in \mathbb{R}^n$,
\begin{eqnarray}\label{a(x)1}
  \left\{\begin{array}{r@{}c@{}ll}
&&D_{\eta} a(x,\eta )\lambda \cdot \lambda \geqslant l\dfrac{g(|\eta|)}{|\eta|}|\lambda|^2 \,, \\[0.05cm]
&&|a(x,\eta)|+|\eta||D_{\eta} a(x,\eta )|\leqslant Lg(|\eta|)\,, \\[0.05cm]
  \end{array}\right.
\end{eqnarray}
where $D_{\eta}$ denotes the differentiation in $\eta$ and $g(t) : [0,+\infty)\rightarrow [0,+\infty)$ satisfies
\begin{eqnarray}\label{a(x)3}
  \left\{\begin{array}{r@{}c@{}ll}
&&g(t)=0 \ \ \ \Leftrightarrow \ \ \   t=0 \,, \\[0.05cm]
&&g(\cdot)\in C^{1}(\mathbb{R}^+)\,, \\[0.05cm]
&& 1\leq i_{g}=: \inf_{t>0}\frac{tg'(t)}{g(t)}\leq \sup_{t>0}\frac{tg'(t)}{g(t)}=:s_{g}<\infty. \, \\[0.05cm]
  \end{array}\right.
\end{eqnarray}
We define
\begin{equation}\label{g}
G(t):= \int_0^tg(\tau)\operatorname{d}\!\tau \ \ \ \mbox{for}\ \ t\geq0.
\end{equation}
It is straightforward to see  that  $G(t)$ is convex and  strictly  increasing. The standard example for $G(\cdot)$ is $$G(t)=\int_{0}^{t}(\mu+s^2)^{\frac{p-2}{2}} s ds$$ with $\mu\geqslant0, p\geqslant2$, then \eqref{a(x)1} is reduced to $p$-growth condition(see \cite{dm11,km12}).

The double obstacle condition that we impose on the solution is  of the form $ \psi_1\leqslant u \leqslant  \psi_2$ a.e.  in $\Omega$, where $\psi_1, \psi_2 \in W^{1,G}(\Omega)\cap W^{2,1}(\Omega)$ ) are two given  functions,where $G$  and  $W^{1,G}(\Omega)$ are defined as \eqref{g} and Definition \ref{class} respectively. Let $(W^{1,G}(\Omega))'$ is the dual of $W^{1,G}(\Omega)$ and  $f \in L^{1}(\Omega)\cap (W^{1,G}(\Omega))'$. In the classical setting, it is natural to consider the double obstacle problem that can be formulated by the variational inequality
\begin{equation}\label{fjd}
\int_{\Omega} a(x,Du)\cdot D(v-u)dx \geq \int_{\Omega} f(v-u) dx
\end{equation}
for all functions $v \in u+W_0^{1,G}(\Omega)$ that  satisfy $ \psi_1\leqslant v \leqslant  \psi_2$ a.e. in $\Omega$. This type of obstacle problem was motivated by many applications arising from engineering, physics, mechanics  and other fields of applied sciences. However,  we are  interested in solutions to the double obstacle problems with measure data in the sense that we want to replace the inhomogeneity $f$ by a bounded Radon measure $\mu$.
 These solutions can be  obtained through approximation by using solutions to the variational inequalities \eqref{fjd}. The precise definition of approximable solutions is provided in Definition \ref{opdy}.

The construction of the nonlinear operator $a(\cdot,\cdot)$ we discussed above is based on the model case
\begin{equation*}
  -\operatorname{div}\left(\kappa(x)\frac{g(|Du|)}{|Du|}Du\right) =\mu \quad\quad\mbox{in}\ \ \ \Omega \\[0.05cm]
\end{equation*}
where $\kappa :\Omega\rightarrow[c,+\infty) $ is  bounded measurable and separated from zero function, and $g$ satisfies \eqref{a(x)3}.
This type of equation is arising in the fields of fluid dynamics, magnetism, and mechanics, see for instance \cite{ber1}. As a class of generalized $p$-Laplacian type elliptic equations, it was first proposed by Lieberman\cite{l1} and moreover, he proved $C^{\alpha}$- and $C^{1,\alpha}$-regularity of the solutions for these  elliptic equations in his paper.
Since then, substantial advancements have been made in the regularity theory of such equations. For a comprehensive overview of the topic we recommend referring to these articles \cite{bm20,cm16,cm17,cm15,l2,rt1,xiong3,xiong6,zz1}.

 In this work, our main goal is to  derive pointwise and oscillation estimates for the gradients of solutions to double obstacle problems by the nonlinear Wolff potentials.  The Wolff potential was introduced by Maz'ya and Havin \cite{mh31} and the relevant important contributions were attributed to Hedberg and Wolff \cite{hw32}.   The fundamental results due to  Kilpel$\ddot{a}$inen and Mal$\acute{y}$ \cite{km5,km6} consist of  the pointwise estimates of solutions to the nonlinear equations of $p$-Laplace type by the   Wolff potentials. Later these results have been extended to a general setting by Trudinger and Wang \cite{tw7,tw8}  by means of  a different  approach. Further remarkable results for the gradient of solutions have been achieved by  Duzaar, Kuusi and Mingione \cite{dm10,dm11,km12,m9}.
 Potential estimates are comprehensive tools to obtain regularity estimates. Indeed, they have been intensively studied and extended in several directions; we refer to
\cite{dm21,km13,km21} for the elliptic systems,  \cite{km22,km23} for the parabolic equations,
  \cite{s26,s28} for single obstacle problems,   respectively.
For supplementary results, please see \cite{km15,km14,km24,km25,mz1,xiong5}.

As for the elliptic equations with Orlicz growth, Baroni \cite{b13}  obtained  pointwise   gradient estimates for solutions of  equations with constant coefficients  by the nonlinear potentials.  Later, these results were upgraded by Xiong and Zhang and Ma in \cite{xiong1,xiong2} to single obstacle   problems with measure data and Orlicz growth. In this paper we consider double obstacle problems. To overcome the difficulty in this case, we establish a  suitable comparison estimate to transfer the double obstacle problems to the single one, so that we can still apply the analytic tools in \cite{km12,m9}.
Initially, we deduce excess decay estimates for solutions of double obstacle problems through the utilization of  certain comparison estimates.  Subsequently, through iterative procedures based on the obtained estimates, we derive pointwise estimates for fractional maximal operators. Ultimately,  these estimates allow us to draw conclusions about pointwise and oscillation estimates for the gradients of solutions.

Next, we summarize our main results. We begin by presenting some  definitions, notations and assumptions.
\begin{definition}
A function $G :[0,+\infty)\rightarrow[0,+\infty)$ is called a Young function if it is convex and $G(0)=0$.
\end{definition}
\begin{definition}\label{class}
Assume that G is a Young function,  the Orlicz class $K^{G}(\Omega)$ is the set of all measurable functions $u : \Omega\rightarrow\mathbb{R}$ satisfying
\begin{equation*}
\int_\Omega G(|u|) \operatorname{d}\!\xi < \infty.\nonumber
\end{equation*}
The Orlicz space $L^{G}(\Omega)$ is the linear hull  of the Orlicz class  $K^{G}(\Omega)$ with the Luxemburg norm
\begin{equation*}
\Vert u \Vert_{L^G(\Omega)}:=\inf\left\lbrace \alpha>0: \ \ \int_{\Omega}G\left(\frac{|u|}{\alpha} \right) \operatorname{d}\!\xi \leqslant1\right\rbrace .
\end{equation*}
Furthermore, the Orlicz-Sobolev space $W^{1,G}(\Omega)$ is defined as
\begin{equation*}
W^{1,G}(\Omega)=\left\lbrace  u\in L^{G}(\Omega)\cap W^{1,1}(\Omega) \ \vert \ Du\in L^{G}(\Omega)\right\rbrace.\nonumber
\end{equation*}
The space $W^{1,G}(\Omega)$, equipped with the norm
$\Vert u \Vert_{W^{1,G}(\Omega)}:=\Vert u \Vert_{L^G(\Omega)}+\Vert Du \Vert_{L^G(\Omega)},$ is a Banach space. Clearly, $W^{1,G}(\Omega)=W^{1,p}(\Omega)$, the standard Sobolev space, if $G(t)=t^p$ with $p\geqslant1$.
\end{definition}
Note that for the Luxemburg norm there holds the inequality
\begin{equation*}
\Vert u \Vert_{L^G(\Omega)}\leqslant \int_{\Omega}G(|u|) \operatorname{d}\!\xi +1.
\end{equation*}

 The subspace $W_{0}^{1,G}(\Omega)$ is the closure of $C_{0}^{\infty}(\Omega)$ in $W^{1,G}(\Omega)$. The above properties about Orlicz space can be found in \cite{rr28}.

For every $k>0$ we let
\begin{equation*}
T_{k}(s):=
\left\{\begin{array}{r@{\ \ }c@{\ \ }ll}
s\ \ \ \ \ \ \ \ if\ \ |s|\leqslant k\,, \\[0.05cm]
k\ sgn(s)\ \ \ \ \ \ \ if\ \ |s|> k\,. \\[0.05cm]
\end{array}\right.
\end{equation*}

Moreover, for given Dirichlet boundary data $h\in W^{1,G}(\Omega)$, we define
$$\mathcal{T}^{1,G}_{h}(\Omega):=\left\lbrace u: \Omega\rightarrow \mathbb{R} \ measurable: T_{k}(u-h)\in W_{0}^{1,G}(\Omega) \ \ for \ all \ k>0\right\rbrace. $$
We now give the definition of approximable solutions.

\begin{definition}\label{opdy}
Suppose that  obstacle functions $\psi_1, \psi_2 \in W^{1,G}(\Omega)$, measure data $\mu \in \mathcal{M}_{b}(\Omega)$ and boundary data $h \in W^{1,G}(\Omega)$ with $ \psi_1\leqslant h \leqslant  \psi_2$ a.e. are given. We say that $u \in \mathcal{T}^{1,G}_{h}(\Omega)$ with $ \psi_1\leqslant u \leqslant  \psi_2$ a.e. on $\Omega$ is  a limit of approximating solutions of the double obstacle problem $OP(\psi_1, \psi_2; \mu)$ if there exist functions
$$f_{i} \in (W^{1,G}(\Omega))'\cap L^{1}(\Omega)\ \  with\ \  f_{i}\stackrel{\ast}\rightharpoonup \mu \ in \ \mathcal{M}_{b}(\Omega) \ \ as \ i\rightarrow+\infty$$
 satisfies
$$\limsup_{i\rightarrow+\infty}\int_{B_R(x_0)}|f_i|dx\leqslant|\mu|(\overline{B_R(x_0)})$$
and solutions $u_{i}\in W^{1,G}(\Omega)$ with $ \psi_1\leqslant u_i \leqslant  \psi_2$ of the variational inequalities
\begin{equation}\label{opdy1}
\int_{\Omega}a(x,Du_{i})\cdot D(v-u_{i})dx\geqslant \int_{\Omega}f_{i}(v-u_{i})dx
\end{equation}
for $\forall \ v \in u_{i}+W_{0}^{1,G}(\Omega)$ with $ \psi_1\leqslant v \leqslant  \psi_2$ a.e. in $\Omega$, such that for $i\rightarrow +\infty$,
$$u_{i}\rightarrow u \ \ a.e. \ \ \  in \ \  \Omega$$
and $$u_{i}\rightarrow u \ \ \ in \ \ \ W^{1,1}(\Omega).$$
\end{definition}
The existence of approximating solutions, converging as defined above, has been proved in our previous work \cite{xiong1} for the inequalities \eqref{opdy1} with a single obstacle. Consequently, the existence in this paper can be attained through minor modifications.

Let us next  turn our attention to  the classical non-linear Wolff potential which is defined by
\begin{equation*}
W^{\mu}_{\beta,p}(x,R):=\int_0^R\left( \frac{|\mu|(B_{\rho}(x))}{\rho^{n-\beta p}}\right) ^{1/(p-1)}\frac{\operatorname{d}\!\rho}{\rho}
\end{equation*}
for parameters $\beta \in (0,n]$  and $p>1$.  We also abbreviate
\begin{equation*}
W^{[\psi_1]}_{\beta,p}(x,R):=\int_0^R\left( \frac{D\Psi_1(B_{\rho}(x))}{\rho^{n-\beta p}}\right) ^{1/(p-1)}\frac{\operatorname{d}\!\rho}{\rho}
\end{equation*}
and
\begin{equation*}
W^{[\psi_2]}_{\beta,p}(x,R):=\int_0^R\left( \frac{D\Psi_2(B_{\rho}(x))}{\rho^{n-\beta p}}\right) ^{1/(p-1)}\frac{\operatorname{d}\!\rho}{\rho}
\end{equation*}
with $D\Psi_1(B_{\rho}(x)):=\int_{B_{\rho}(x)}\left(\frac{g(|D\psi_1|)}{|D\psi_1|}|D^{2}\psi_1|+1\right)d\xi$, $D\Psi_2(B_{\rho}(x)):=\int_{B_{\rho}(x)}| \ddiv a(x,D\psi_2)|d\xi$.

  Now we recall the definitions of the centered  maximal operators as follows.
\begin{definition}
Let $ \beta\in[0,n], x\in \Omega$ and $ R<dist(x,\partial\Omega) $, and let $u$ be an $ L^1(\Omega) $-function or a measure with finite mass; the restricted fractional $ \beta $ maximal function of  $u$ is defined by
\begin{equation*}
M_{\beta,R}(u)(x):=\sup_{0<r\leqslant R}r^{\beta}\frac{|u|(B_r(x))}{|B_r(x)|}=\sup_{0<r\leqslant R}r^{\beta}\fint_{B_r(x)}|u|\operatorname{d}\!\xi.
\end{equation*}
\end{definition}
Note that when $ \beta=0 $ the one defined above is the classical Hardy-Littlewood maximal operator.

Moreover, we define
\begin{equation*}
\overline{M}_{\beta,R}(\psi_1)(x):=\sup_{0<r\leqslant R}r^{\beta}\frac{D \Psi_1(B_r(x))}{|B_r(x)|}=\sup_{0<r\leqslant R}r^{\beta}\fint_{B_r(x)}\left(\frac{g(|D\psi_1|)}{|D\psi_1|}|D^{2}\psi_1|+1\right)\operatorname{d}\!\xi
\end{equation*}
and
\begin{equation*}
\overline{M}_{\beta,R}(\psi_2)(x):=\sup_{0<r\leqslant R}r^{\beta}\frac{D \Psi_2(B_r(x))}{|B_r(x)|}=\sup_{0<r\leqslant R}r^{\beta}\fint_{B_r(x)} | \ddiv a(x,D\psi_2)|\operatorname{d}\!\xi.
\end{equation*}
\begin{definition}
Let $ \beta\in[0,n], x\in \Omega$ and $ R<dist(x,\partial\Omega) $, and let $u$ be an $ L^1(\Omega) $-function or a measure with finite mass; the restricted sharp  fractional $ \beta $ maximal function of  $u$ is defined by
\begin{equation*}
M^{\#}_{\beta,R}(u)(x):=\sup_{0<r\leqslant R}r^{-\beta}\fint_{B_r(x)}|u-(u)_{B_r(x)}|\operatorname{d}\!\xi.
\end{equation*}
\end{definition}
When $ \beta=0 $ the one defined above is the Fefferman-Stein sharp maximal operator.

Throughout this paper we write
$$\theta(a,B_{r}(x_0))(x):=\sup_{\eta \in \mathbb{R}^n\setminus \left\lbrace0\right\rbrace  }\frac{|a(x,\eta)-\overline{a}_{B_{r}(x_0)}(\eta)|}{g(|\eta|)}, $$
where $$\overline{a}_{B_{r}(x_0)}(\eta):=\fint_{B_{r}(x_0)}a(x,\eta)dx.$$
Then we can easily check from \eqref{a(x)1} that $|\theta(a,B_{r}(x_0))|\leqslant2L$.
\begin{definition}\label{df1}
We say that  $a(x,\eta)$ is ($\delta$, R)-vanishing for some $\delta, R>0$,  if
\begin{equation}\label{a(x)2}
\omega(R):=\sup_{{\substack{ x_{0}\,\in\,\Omega\\0<r\leq R}}}  \fint_{B_{r}(x_{0})}\theta(a,B_{r}(x_{0}))\operatorname{d}\!x \leq\delta.
\end{equation}
Moreover, ${\omega(\cdot)}^{\frac{1}{1+s_g}}$ will be called Dini-BMO regular if
\begin{equation*}
\sup_{r>0}\int_{0}^{r}{\omega(\rho)}^{\frac{1}{1+s_g}}\frac{d\rho}{\rho}< +\infty;
\end{equation*}
${\omega(\cdot)}^{\frac{1}{1+s_g}}$ will be called Dini-H$\ddot{o}$lder regular if
\begin{equation*}
\sup_{r>0}\int_{0}^{r}\frac{{\omega(\rho)}^{\frac{1}{1+s_g}}}{\rho^{\alpha}}\frac{d\rho}{\rho}< +\infty. \ \ \ \
\end{equation*}
\end{definition}
Throughout this paper, we always assume that $\delta$ is a small positive constant.

Finally we state our main results of this paper.
The first result we are going to present is some pointwise estimates of certain maximal operators of approximable solutions.
\begin{theorem}\label{Th1}
Under the assumptions  \eqref{a(x)1},  \eqref{a(x)3} and   \eqref{a(x)2}, let $\psi_1, \psi_2 \in W^{1,G}(\Omega)\cap W^{2,1}(\Omega)$,  $ \ddiv a(x,D\psi_2),  \frac{g(|D\psi_1|)}{|D\psi_1|}|D^{2}\psi_1| \in L^{1}_{loc}(\Omega) $. Assume that $u \in W^{1,1}(\Omega)$ with $\psi_1 \ls u \ls \psi_2$ a.e. is a limit of approximating solutions to $OP(\psi_1, \psi_2; \mu)$ with measure data $\mu \in \mathcal{M}_{b}(\Omega)$(in the sense of Definition \ref{opdy}),  and ${\omega(\cdot)}^{\frac{1}{1+s_g}}$ is Dini-BMO regular in the sense of Definition \ref{df1},
then  there exists a constant $c=c(data) $ and a radius $R_{0}>0$, depending on $data,\omega(\cdot)$,  such that
\begin{eqnarray}\label{1.11}
\nonumber &&M_{\alpha,R}^{\#}(u)(x)+M_{1-\alpha,R}(Du)(x)\\ \nonumber
&\leqslant& c\left[ W_{1-\alpha+\frac{\alpha}{i_g+1},i_g+1}^{\mu}(x,2R)+W_{1-\alpha+\frac{\alpha}{i_g+1},i_g+1}^{[\psi_1]}(x,2R)+W_{1-\alpha+\frac{\alpha}{i_g+1},i_g+1}^{[\psi_2]}(x,2R)\right]  \\
&+&cR^{1-\alpha}\fint_{B_R(x)}|Du|d\xi+c\int_{0}^{2R}[\omega(\rho)]^{\frac{1}{1+s_g}}G^{-1}\left[\fint_{B_{\rho}(x)}[G(|D\psi_1|)+G(|\psi_1|)]d\xi \right]\frac{d\rho}{\rho^{\alpha}}.
\end{eqnarray}
Further assume that
\begin{equation}\label{wtiaojian}
\sup_{r>0}\frac{[\omega(r)]^{\frac{1}{1+s_g}}}{r^{\widehat{\alpha}}}\leqslant c_0,
\end{equation}
for some $\widehat{\alpha}\in[0,\beta)$, then
\begin{eqnarray}\label{1.122}
\nonumber &&M_{\alpha,R}^{\#}(Du)(x) \\ \non
&\leqslant& c\left\lbrace  \left[  M_{1-\alpha i_g,R}(\mu)(x)\right] ^{\frac{1}{i_g}}+\left[  \overline{M}_{1-\alpha i_g,R}(\psi_1)(x)\right] ^{\frac{1}{i_g}}+\left[  \overline{M}_{1-\alpha i_g,R}(\psi_2)(x)\right] ^{\frac{1}{i_g}}\right\rbrace  \\ \nonumber
&+&c\left[ W_{\frac{1}{i_g+1},i_g+1}^{\mu}(x,2R)+W_{ \frac{1}{i_g+1},i_g+1}^{[\psi_1]}(x,2R)+W_{ \frac{1}{i_g+1},i_g+1}^{[\psi_2]}(x,2R)\right]   \\
&+&cR^{-\alpha}\fint_{B_R(x)}|Du|\operatorname{d}\! \xi+c\int_{0}^{2R}[\omega(\rho)]^{\frac{1}{1+s_g}}G^{-1}\left[\fint_{B_{\rho}(x)}[G(|D\psi_1|)+G(|\psi_1|)]d\xi \right]\frac{d\rho}{\rho^{1+\alpha}}
\end{eqnarray}
holds uniformly in $\alpha\in[0,\widehat{\alpha}]$,
where $c=c(data,\widehat{\alpha},c_0,\omega(\cdot),diam(\Omega))$, $ \ \ \ \ \\ $$0<R\leqslant  \min \left\lbrace R_0,dist(x_0,\partial \Omega)\right\rbrace $ and $\beta$ is as in Lemma \ref{lemma1.6}.
\end{theorem}

  The previous theorem implies the  pointwise and oscillation estimates for the gradients of solutions to double obstacle problems.

\begin{theorem}\label{Th2}
In the same hypotheses of Theorem \ref{Th1}, let $B_{4R}(x_0)\subseteq \Omega, x,y \in B_{\frac{R}{4}}(x_0), 0<R\leqslant\frac{1}{2},$ and for some $\widehat{\alpha}\in[0,\beta)$, assume that ${\omega(\cdot)}^{\frac{1}{1+s_g}}$ is  Dini-H$\ddot{o}$lder regular, that is
\begin{equation}
c_0:=\sup_{r>0}\int_{0}^{r}\frac{[\omega(\rho)]^{\frac{1}{1+s_g}}}{\rho^{\widehat{\alpha}}}\frac{d\rho}{\rho}<+\infty.
\end{equation}
Then  there exists a constant $c=c(n,i_g,s_g,v, L, c_0,  \widehat{\alpha},\omega(\cdot),diam(\Omega)) $ such that
\begin{eqnarray} \label{du}\nonumber 
\nonumber &&|Du(x_0)|  \\\nonumber 
&\leqslant& c\left[ \fint_{B_R(x_0)}|Du|d\xi+W_{\frac{1}{i_g+1},i_g+1}^{\mu}(x_0,2R)+W_{ \frac{1}{i_g+1},i_g+1}^{[\psi_1]}(x_0,2R)+W_{ \frac{1}{i_g+1},i_g+1}^{[\psi_2]}(x_0,2R)\right] \\
&+&c\int_{0}^{2R}[\omega(\rho)]^{\frac{1}{1+s_g}}G^{-1}\left[\fint_{B_{\rho}(x_0)}[G(|D\psi_1|)+G(|\psi_1|)]d\xi \right]\frac{d\rho}{\rho},
\end{eqnarray}
\begin{eqnarray}\label{du-du} \nonumber
&&\vert Du(x)-Du(y) \vert \\ \nonumber
&\leq &c\fint_{B_R(x_0)}\vert Du\vert\operatorname{d}\!\xi\left(\frac{|x-y|}{R} \right) ^{\alpha}  \\ \nonumber
&+& c \left[  W^{\mu}_{-\alpha+\frac{1+\alpha}{1+i_g},i_g+1}(x,2R)
+ W^{[\psi_1]}_{-\alpha+\frac{1+\alpha}{1+i_g},i_g+1}(x,2R)+W^{[\psi_2]}_{-\alpha+\frac{1+\alpha}{1+i_g},i_g+1}(x,2R)\right]|x-y|^{\alpha} \\ \nonumber
&+&c \left[  W^{\mu}_{-\alpha+\frac{1+\alpha}{1+i_g},i_g+1}(y,2R)
+ W^{[\psi_1]}_{-\alpha+\frac{1+\alpha}{1+i_g},i_g+1}(y,2R)+ +W^{[\psi_2]}_{-\alpha+\frac{1+\alpha}{1+i_g},i_g+1}(y,2R)\right]|x-y|^{\alpha} \\ \nonumber
&+&c \left[ \int_{0}^{2R}[\omega(\rho)]^{\frac{1}{1+s_g}}G^{-1}\left[\fint_{B_{\rho}(x)}[G(|D\psi_1|)+G(|\psi_1|)]d\xi \right]\frac{d\rho}{\rho^{1+\alpha}} \right]|x-y|^{\alpha} \\
&+&c \left[  \int_{0}^{2R}[\omega(\rho)]^{\frac{1}{1+s_g}}G^{-1}\left[\fint_{B_{\rho}(y)}[G(|D\psi_1|)+G(|\psi_1|)]d\xi \right]\frac{d\rho}{\rho^{1+\alpha}}\right]|x-y|^{\alpha},
\end{eqnarray}
holds uniformly in $\alpha\in[0,\widehat{\alpha}]$, where  $\beta$ is as in Lemma \ref{lemma1.6}  and $x,y$ is the Lebesgue's point of $Du$.
\end{theorem}

\begin{remark}
 To the best of our knowledge, there is not much research on the  gradient estimate associated with  the double obstacle problems, and our work is new,  and our results provide a new perspective  for understanding  the solutions to the double obstacle problems.
\end{remark}
The remainder of this paper is organized as follows. Section 2 contains some notions and preliminary results. In Section 3,   we  derive  the  excess decay estimate for solutions to these problems by using some comparison estimates.  In Section 4, we obtain  pointwise and oscillation estimates for the gradients of solutions  by the  Wolff potentials.

\section{Preliminaries}\label{section2}
In this section, we introduce some notions and results which will be used in this paper. For brevity in notation, we gather the dependencies of specific constants on the parameters of our inquiry as $$data = data(n,i_g,s_g,l,L).$$
For an integrable map $f: \Omega \rightarrow \mathbb{R}^n $, we write 
$$(f)_{\Omega}:=\fint_{\Omega}fdx:=\frac{1}{|\Omega|}\int_{\Omega}fdx.$$
For $q\in[1,\infty)$, it is easily verified that
\begin{equation} \label{1.8}
\parallel f-(f)_{\Omega}\parallel_{L^q(\Omega)}\leqslant2\min_{c\in \mathbb{R}^m}\parallel f-c\parallel_{L^q(\Omega)}.
\end{equation}

\begin{definition}
A Young function $G$ is called an $N$-function if
$$0<G(t)<+\infty \ \ for \ t>0$$
and
\begin{equation}\label{nhanshu}
\lim_{t\rightarrow+\infty}\frac{G(t)}{t}=\lim_{t\rightarrow0}\frac{t}{G(t)}=+\infty.
\end{equation}
It's obvious that $G(t)$ is an $N$-function.

The Young conjugate  of a Young function G will be denoted by $G^{\ast}$ and defined as
$$G^{\ast}(t)=\sup_{s\geq 0}\left\lbrace st-G(s)\right\rbrace  \ \ for \ t\geq 0.$$
\end{definition}
In particular,  if $G$ is an $N$-function, then $G^{\ast}$ is  an $N$-function as well.
\begin{definition}
A Young function G is said to satisfy the global $\vartriangle_2$ condition, denoted by $G\in\vartriangle_2$, if there exists a positive constant C such that for every $t>0$,
\begin{equation*}
G(2t)\leq CG(t).
\end{equation*}
Similarly, a Young function G is said to satisfy the global $\bigtriangledown_2$ condition, denoted by $G\in\bigtriangledown_2$, if there exists a  constant $\theta >1$ such that for every $t>0$,
\begin{equation*}
G(t)\leq \frac{G(\theta t)}{2\theta}.
\end{equation*}
\end{definition}

\begin{remark}  \label{remark1}
For an increasing function $f: \mathbb{R}^+\rightarrow\mathbb{R}^+$ satisfying  $\vartriangle_2$ condition $f(2t)\lesssim f(t)$ for $t\geqslant0$, it is easy to prove that $f(t+s)\leqslant c[f(t)+f(s)]$ holds for every $t,s\geqslant0$.
\end{remark}

Next let us recall a basic property of $N$-function, which will be used in the sequel.
\begin{lemma}\cite{yz1}\label{gyoung}
If $G$ is an $N$-function, then $G$ satisfies the following Young's inequality
$$st\leq G^{*}(s)+G(t), \ \ \ for \ \ \forall s,t\geq0.$$
Furthermore, if $G\in \bigtriangleup_{2}\cap \bigtriangledown_{2}$ is an $N$-function, then $G$ satisfies the following Young's inequality with $\forall \varepsilon >0$,
$$st\leq \varepsilon G^{*}(s)+c(\varepsilon)G(t), \ \ \ for \ \ \forall s,t\geq0.$$
\end{lemma}

Another important property of Young's conjugate function is the following inequality, which  can be found in \cite{a1}:
\begin{equation}\label{a(x)4}
G^{*}\left( \frac{G(t)}{t}\right) \leqslant G(t).
\end{equation}
\begin{lemma}\cite{cm17,yz1}\label{gwan}
Under the assumption  \eqref{a(x)3}, G(t) is defined in \eqref{g}. Then we have

(1) $G(t)$ is strictly convex $N$-function and
$$G^{\ast}(g(t))\leqslant c G(t) \ \ \ for \ t\geqslant0 \ \ \ and \ \ some\ \  c>0;$$

(2) $G(t) \in\bigtriangledown _{2}.$
\end{lemma}

\begin{lemma}\cite{cho1,de1}\label{adaog}
Under the assumptions  \eqref{a(x)1} and \eqref{a(x)3},  G(t) is defined in \eqref{g}. Then there exists $c=c(data)>0 $ such that
\begin{equation}
[a(x,\eta)-a(x,\xi)]\cdot(\eta-\xi)\geq cG(|\eta-\xi|), \ \ \ \ for \ \ every \ \ x \in \Omega,\eta,\xi , \in \mathbb{R}^n.
\end{equation}
Especially, we have
\begin{equation}
a(x,\eta) \cdot \eta \geq cG(|\eta|), \ \ \ \ \ for \ \ every \ \ x \in \Omega,\eta \in \mathbb{R}^n.
\end{equation}
\end{lemma}
The following iteration lemma turns out to be very useful in the sequel.
\begin{lemma}\cite{g11} \label{diedai}
 Let $f(t)$ be a nonnegative function defined on the interval $[a,b]$ with $a \geqslant 0$. Suppose that for $s,t \in [a,b]$ with $t<s$,
 \begin{equation*}
 f(t)\leqslant \frac{A}{(s-t)^{\alpha}}+\frac{B}{(s-t)^{\beta}}+C+\theta f(s)
 \end{equation*}
 holds, where $A,B,C\geqslant0, \alpha,\beta>0$ and $0\leqslant\theta<1$. Then there exists a constant $c=c(\alpha,\theta)$ such that
  \begin{equation*}
 f(\rho)\leqslant c\left( \frac{A}{(R-\rho)^{\alpha}}+\frac{B}{(R-\rho)^{\beta}}+C\right)
 \end{equation*}
 for any $\rho,R \in [a,b]$ with $\rho<R$.
\end{lemma}
The proof of the following lemma can be found in \cite{xiong1}.
\begin{lemma}\label{qianqi}
Let $\Omega \subset \mathbb{R}$ be a bounded domain. Assume that $1+i_g\leqslant n$ , $h \in \mathcal{T}^{1,G}_{0}(\Omega)$ satisfies
\begin{equation*}
\int_{\Omega \cap \left\lbrace |h|\leqslant k\right\rbrace }|Dh|^{1+i_g}dx\leqslant Mk+M^{\frac{1+i_g}{i_g}}
\end{equation*}
 for $\forall \  k>0$,  and fixed constants  $M>0$.
Then we have
$$\int_{\Omega}|h|^{1+\alpha}dx\leqslant c_{1}M^{\frac{1+\alpha}{i_g}},$$
$$\int_{\Omega}|Dh|^{1+\beta}dx\leqslant c_{2}M^{\frac{1+\beta}{i_g}},$$
where $0<\alpha <min\left\lbrace 1,\frac{n(i_g-1)+1+i_g}{n-1-i_g}\right\rbrace, 0<\beta <min\left\lbrace 1,\frac{n(i_g-1)+1}{n-1}\right\rbrace, c_1=c_1(\Omega,n,i_g,\alpha),$ $c_2=c_2(\Omega,n,i_g,\beta).$
\end{lemma}

 \begin{lemma}\label{ag}
Assume that $g(t)$ satisfies   \eqref{a(x)3}, $G(t)$ is defined in \eqref{g}. Then we have

(1) for any  $\beta \geq1$,

 \ \ \ \ \ \ \ \ \ $\beta^{i_g}\leq \dfrac{g(\beta t)}{g(t)}\leq \beta^{s_g}$ \ \ \ and \ \ \  $\beta^{1+i_g}\leq \dfrac{G(\beta t)}{G(t)} \leq \beta^{1+s_g}$,  \ \ \  for every $t>0$,

for any  $0<\beta<1$,

 \ \ \ \ \ \ \ \ \ $\beta^{s_g}\leq \dfrac{g(\beta t)}{g(t)}\leq \beta^{i_g}$ \ \ \ and \ \ \
$\beta^{1+s_g}\leq \dfrac{G(\beta t)}{G(t)} \leq \beta^{1+i_g}$, \ \ \ for every $t>0$.

(2) for any  $\beta \geq1$,

 \ \ \ \ \ \ \ \ \ $\beta^{\frac{1}{s_g}}\leq \dfrac{g^{-1}(\beta t)}{g^{-1}(t)}\leq \beta^{\frac{1}{i_g}}$ \ \ \ and \ \ \  $\beta^{\frac{1}{1+s_g}}\leq \dfrac{G^{-1}(\beta t)}{G^{-1}(t)} \leq \beta^{\frac{1}{1+i_g}}$,  \ \ \  for every $t>0$,

for any  $0<\beta<1$,

 \ \ \ \ \ \ \ \ \ $\beta^{\frac{1}{i_g}}\leq \dfrac{g^{-1}(\beta t)}{g^{-1}(t)}\leq \beta^{\frac{1}{s_g}}$ \ \ \ and \ \ \
$\beta^{\frac{1}{1+i_g}}\leq \dfrac{G^{-1}(\beta t)}{G^{-1}(t)} \leq \beta^{\frac{1}{1+s_g}}$, \ \ \ for every $t>0$.
\end{lemma}
It's obvious that  Lemma \ref{ag} implies that
\begin{equation}\label{lg}
L^{1+s_g}(\Omega)\subset L^{G}(\Omega) \subset L^{1+i_g}(\Omega) \subset L^1(\Omega),
\end{equation}
and $g$ and $G$ satisfy $\vartriangle_2$ condition. Then from Remark   \ref{remark1}, we know that  $g$ and $G$ satisfy   the subadditivity property: $g(t+s)\leqslant c[g(t)+g(s)]$,  $G(t+s)\leqslant c[G(t)+G(s)], $ for every $t,s \geqslant0$.
\section{Comparison estimates }\label{section3}
This section is devoted to comparing the solutions of double obstacle problems to the solutions of elliptic equations. Therefore we can obtain a similar excess decay estimate for solutions of double obstacle problems with measure data.
Firstly, we shall prove the comparison estimate between the solutions of double obstacle problems with measure data and the solutions of the single obstacle problems.

\begin{lemma}\label{bijiao}
Assume that conditions  \eqref{a(x)1}-\eqref{a(x)3} are fulfilled, let $B_{2R}(x_0)\subset \Omega,  f \in L^{1}(B_R(x_0))\cap (W^{1,G}(B_R(x_0)))'$ and the map $u\in W^{1,G}(B_R(x_0))$ with $ \psi_1 \ls u \ls \psi_2$ solves the variational inequality
\begin{equation}\label{u0}
\int_{B_R(x_0)} a(x,Du)\cdot D(v-u)dx \geq \int_{B_R(x_0)} f(v-u) dx
\end{equation}
for any  $v \in u+W_0^{1,G}(B_R(x_0))$ that  satisfy $\psi_1 \ls v\ls \psi_2 $ a.e.  in $B_R(x_0)$.
Let $w_1 \in u+W_{0}^{1,G}(B_R(x_0))$ with $w_1 \geq \psi_1 $ be the weak solution of the single obstacle problem
\begin{equation}\label{w-v}
\int_{B_R(x_0)} a(x,Dw_1)\cdot D(v-w_1)dx \geq \int_{B_R(x_0)} a(x,D\psi_2)\cdot D(v-w_1)dx
\end{equation}
for any  $v \in w_1+W_0^{1,G}(B_R(x_0))$ that  satisfy $ v \geqslant  \psi_1 $ a.e.  in $B_R(x_0)$.
Then there exists $c=c(data)$ such that
\begin{equation}
\fint_{B_R(x_0)} |Du-Dw_1| \operatorname{d}x\leqslant c\left[R \fint_{B_R(x_0)}|f|+|\ddiv a(x,D\psi_2)|dx\right] ^{\frac{1}{i_g}}.
\end{equation}
\end{lemma}

\begin{proof}
Without loss of generality we may assume that $x_0=0, R=1$ by defining

$$\widehat{u}(x)=\dfrac{u(Rx+x_0)}{R}, \ \ \ \widehat{w_1}(x)=\dfrac{w_1(Rx+x_0)}{R}, \ \ \ \widehat{f}(x)=Rf(Rx+x_0), \ \ \ $$ $$\widehat{\psi_2}(x)=\dfrac{\psi_2(Rx+x_0)}{R}, \ \ \ \widehat{a}(a,\eta)=a(Rx+x_0,\eta).$$
Since $w_1 \in u+W_0^{1,G}(B_1), u \ls \psi_2$ a.e. in $B_1$, we have $(w_1-\psi_2)_+ \in W_0^{1,G}(B_1)$. Then we take $v=\min \{w_1,\psi_2\}=w_1-(w_1-\psi_2)_+ \in  w_1+W_0^{1,G}(B_1)$ with $v \rs \psi_1$ as comparison functions in \eqref{w-v},   it follows from  Lemma \ref{adaog}  that
\begin{eqnarray*}
\int_{B_1}G(|D(w_1-\psi_2)_+|)dx &\leqslant& c \int_{B_1}[a(x,Dw_1)-a(x,D\psi_2)] \cdot D[(w_1-\psi_2)_+]dx \\
&\leqslant& 0.
\end{eqnarray*}
Because $G(\cdot)$ is increasing over $[0,+\infty)$ and $G(0)=0$, we can infer that
$$D(w_1-\psi_2)_+=0 \ \ a.e. \ \ \ in \ \ B_1.$$
Combining with $(w_1-\psi_2)_+=0\ \ \  on \ \ \partial B_1$, which implies
$$(w_1-\psi_2)_+=0 \ \ \ a.e.\ \ in \ \ B_1.$$
It means that
$w_1\ls \psi_2$ a.e. in $B_1$.

Case 1: $\int_{B_1}|f|+|\ddiv a(x,D\psi_2)|dx\leqslant 1$. If $1+i_g>n$, because of $u-w_1\in W_{0}^{1,G}(B_1)$, $u-w_1\in W_{0}^{1,1+i_g}(B_1)$, then we make use of Sobolev's inequality to get $u-w_1\in L^{\infty}(B_{1})$. Now we take $v=\frac{u+w_1}{2}\in u+W_{0}^{1,G}(B_1)$ as comparison functions in the variational inequalities \eqref{u0} and \eqref{w-v}, which implies that
\begin{eqnarray*}
\int_{B_1}|Du-Dw_1|^{1+i_g}dx &\leq & c\int_{B_1}[G(|Du-Dw_1|)+1]dx    \nonumber \\
&\leq & c\int_{B_1}[a(x,Dw_1)-a(x,Du)] \cdot (Dw_1-Du)dx +c \nonumber  \\
&\leq & c\int_{B_1}\left[ |f|+|\ddiv a(x,D\psi_2)|\right] |w_1-u|dx+c    \nonumber  \\
&\leq & c\parallel u-w_1\parallel_{L^{\infty}(B_1)}\int_{B_1}|f|+|\ddiv a(x,D\psi_2)|dx+c \\
&\leq & c \parallel Du-Dw_1\parallel_{L^{1+i_g}(B_1)}+c,
\end{eqnarray*}
where we used Lemma \ref{ag} and Lemma \ref{adaog}.
Then we have
\begin{equation*}
\int_{B_1}|Du-Dw_1|dx\leqslant c.
\end{equation*}
If $1+i_g\leqslant n$, we define
$$D_{k}:=\left\lbrace x \in B_{1} : |u(x)-w_1(x)| \leq k \right\rbrace, \ \ \ \forall \ k>0, $$
and let $v_k:=u+T_{k}(w_1-u)$ and $\overline{v_k}:=w_1-T_{k}(w_1-u) \in u+W_{0}^{1,G}(B_1)$ as comparison functions in the variational inequalities \eqref{u0} and \eqref{w-v} respectively, then  we have
 \begin{equation}\label{k-}
 \int_{B_1}[a(x,Dw_1)-a(x,Du)] \cdot D[T_{k}(w_1-u)]dx \leq k\int_{B_1}|f|+|\ddiv a(x,D\psi_2)|dx.
 \end{equation}
 Then  for every $k\geqslant1$, we have
\begin{eqnarray*}
\int_{D_k}|Du-Dw_1|^{1+i_g}dx
&\leq & c\int_{B_1}[a(x,Dw_1)-a(x,Du)] \cdot D[T_{k}(w_1-u)]dx+c    \nonumber  \\
&\leq & ck.
\end{eqnarray*}
 Now using Lemma \ref{qianqi}, we obtain
$$\int_{B_1}|Du-Dw_1|dx \leqslant c.$$

Case 2: $\int_{B_1}|f|+|\ddiv a(x,D\psi_2)|dx>1$.
We let $$\overline{u}(x)=A^{-1}u(x), \ \ \ \overline{v}(x)=A^{-1}v(x), \ \ \ \overline{w_1}(x)=A^{-1}w_1(x), \ \ \  \overline{\psi_1}(x)=A^{-1}\psi_1(x)$$
$$\overline{\psi_2}(x)=A^{-1}\psi_2(x), \ \ \ \overline{f}(x)=A^{-i_g}f(x), \ \ \ \overline{a}(x,\eta)=A^{-i_g}a(x,A\eta), \ \ \ \overline{G}(t)=\int_{0}^{t}\overline{g}(\tau)d\tau,$$
where $A=(\int_{B_1}|f|+|\ddiv a(x,D\psi_2)|dx)^{\frac{1}{i_g}}>1$.
Then we can easily obtain $\int_{B_1}|\overline{f}|+|\ddiv a(x,D\psi_2)|dx \ls 1$, and $\overline{a}$ satisfies \eqref{a(x)1} and $\overline{g}$ satisfies \eqref{a(x)3}.

Moreover,
$\overline{u}\in W^{1,G}(B_1)$ with $ \overline{\psi_1} \ls \overline{u}  \ls \overline{\psi_2} $ solves the variational inequality
\begin{equation*}
\int_{B_1} \overline{a}(x,D\overline{u})\cdot D(\overline{v}-\overline{u})dx \geq \int_{B_1} (\overline{v}-\overline{u}) \overline{f}dx
\end{equation*}
for any  $\overline{v} \in \overline{u}+W_0^{1,G}(B_1)$ that  satisfy $ \overline{\psi_1} \ls \overline{v}  \ls \overline{\psi_2} $ a.e.  on $B_1$.

And
$\overline{w_1}\in \overline{u}+W_{0}^{1,G}(B_1)$ with $\overline{w_1} \geq \overline{\psi_1}$ solves the  inequality
\begin{equation*}
\int_{B_1} \overline{a}(x,D\overline{w_1})\cdot D(\overline{v}-\overline{w_1})dx \geq \int_{B_1} \overline{a}(x,D\overline{\psi_2})\cdot D(\overline{v}-\overline{w_1})dx.
\end{equation*}

Similar to the case 1 we have
\begin{equation*}
\int_{B_1} |D\overline{u}-D\overline{w_1}|dx \leqslant c.
\end{equation*}
In conclusion, we have
\begin{equation*}
\int_{B_1}|Du-Dw_1|dx \leqslant c\left(\int_{B_1}|f|+|\ddiv a(x,D\psi_2)|dx\right)^{\frac{1}{i_g}},
\end{equation*}
which finishes our proof.
\end{proof}

\begin{corollary}\label{coro}
Assume that conditions  \eqref{a(x)1}-\eqref{a(x)3} are fulfilled, let $w_1$ be as in Lemma \ref{bijiao} and $\mu \in  \mathcal{M}_{b}(\Omega)$ and u be a limit of approximating solutions for $OP(\psi_1, \psi_2; \mu)$, in the sense of Definition \ref{opdy}. Then there exists $c=c(data)$ such that
\begin{equation} \label{coro1}
\fint_{B_R(x_0)}|Du-Dw_1|dx\leqslant c\left[\frac{|\mu|(\overline{B_R}(x_0))}{R^{n-1}}+R \fint_{B_R(x_0)}|\ddiv a(x,D\psi_2)|dx\right]^{\frac{1}{i_g}}.
\end{equation}
\end{corollary}
The proof of corollary \ref{coro} follows a similar approach to the proof of corollary 4.2 presented in \cite{xiong1}; hence, we omit its repetition in this exposition.

The following results have been proved in our previous work \cite{xiong1,xiong2}.
\begin{lemma}\label{w1-w2}
Assume that conditions  \eqref{a(x)1}-\eqref{a(x)3} are fulfilled, let $B_{2R}(x_0)\subset \Omega,  f \in L^{1}(B_R(x_0))\cap (W^{1,G}(B_R(x_0)))'$ and the map $w_1\in W^{1,G}(B_R(x_0))$ with $ w_1 \geq \psi_1$ solves the variational inequality \eqref{w-v}. 
Let $w_2 \in w_1+W_{0}^{1,G}(B_R(x_0))$ with $w_2 \geq \psi_1 $ be the weak solution of the homogeneous obstacle problem
\begin{equation}\label{vw1}
\int_{B_R(x_0)}a(x,Dw_2)\cdot D(v-w_2)dx\geqslant 0
\end{equation}
for any $v \in w_2+W_0^{1,G}(B_R(x_0))$ with $v\geqslant \psi_1 \ a.e. \ in \ B_{R}(x_0)$. 
Then there exists $c=c(data)$ such that
\begin{equation*}
\fint_{B_R(x_0)} |Dw_1-Dw_2| \operatorname{d}x\leqslant c\left[R \fint_{B_R(x_0)}|\ddiv a(x,D\psi_2)|dx\right] ^{\frac{1}{i_g}}.
\end{equation*}
\end{lemma}

\begin{lemma} \label{diyi}
Under the conditions  \eqref{a(x)1}-\eqref{a(x)3}, let $B_{2R}(x_0)\subseteq \Omega$, $w_2\geqslant 0$, $ \psi_1 $ $\in W^{1,G}(B_{2R}(x_0))$,  $w_2 \in W^{1,G}(B_{2R}(x_0))$ with $w_2\geqslant \psi_1$ solve the inequality
 \eqref{vw1}
  and  $w_3 \in W^{1,G}(B_{2R}(x_0))$ with $w_3 \geqslant \psi_1$ solve the inequality
\begin{equation}\label{vw2}
\int_{B_R(x_0)}\overline{a}_{B_R(x_0)}(Dw_3)\cdot D(v-w_3)dx\geqslant 0
\end{equation}
and $w_3=w_2$ on $\partial B_{R}(x_0)$.
Then we have
\begin{eqnarray*}
 &&\fint_{B_R(x_0)}|Dw_2-Dw_3|dx  \\ 
 &\leqslant&  c \omega(R)^{\frac{1}{1+s_g}}\left\lbrace \fint_{B_{2R}(x_0)}|Dw_2|dx+G^{-1}\left[\fint_{B_{2R}(x_0)}[G(|D\psi_1|)+G(|\psi_1|)]dx \right] \right\rbrace,
\end{eqnarray*}
where $c=c(data)$.
\end{lemma}

\begin{lemma}\label{zabj}
Assume that conditions  \eqref{a(x)1}-\eqref{a(x)3} are fulfilled, and  $z \in W^{1,G}(B_R(x_0))$ with $z\geqslant \psi_1$ solves the inequality \eqref{vw1}.  Let $w \in z+W_0^{1,G}(B_R(x_0))$  be a weak solution of the equation
\begin{equation*}
-\operatorname{div}\left(  a(x, Dw)\right)=  -\operatorname{div}\left(  a(x, D\psi_1)\right) \ \ \ \  in \ \ B_R(x_0)
\end{equation*}
and $ \psi_1 \in W^{1,G}(B_R(x_0))\cap W^{2,1}(B_R(x_0)), \ \operatorname{div}\left(  a(x, D\psi_1)\right) \in L^1(B_R(x_0)).$
Then there exists $c=c(data)>0$ such that
\begin{equation}
\fint_{B_R(x_0)}|Dz-Dw|dx\leqslant c\left(R\fint_{B_R(x_0)}(|\operatorname{div}\left(  a(x, D\psi_1)\right)|+1)dx\right)^{\frac{1}{i_g}}.
\end{equation}
\end{lemma}

\begin{lemma}\label{fbjg}
Assume that conditions  \eqref{a(x)1}-\eqref{a(x)3} are fulfilled, let  $f,g \in L^{1}(B_R(x_0))\cap (W^{1,G}(B_R(x_0)))'$ and
$z,w \in W^{1,G}(B_R(x_0))$ with $z-w \in W_{0}^{1,G}(B_R(x_0))$ be weak solutions of
\begin{equation}\label{fhg}
\left\{\begin{array}{r@{\ \ }c@{\ \ }ll}
-\operatorname{div}\left(  a(x, Dz)\right) =f & \mbox{in}\ \ B_R(x_0)\,, \\[0.05cm]
-\operatorname{div}\left(  a(x, Dw)\right) =g & \mbox{in}\ \ B_R(x_0)\,. \\[0.05cm]
\end{array}\right.
\end{equation}
Then the following comparison estimates hold:
\begin{equation}
\fint_{B_R(x_0)}|Dz-Dw|dx\leqslant c\left(R\fint_{B_R(x_0)}(|f|+|g|+1)dx \right)^{\frac{1}{i_g}},
\end{equation}
where $c=c(data)>0$.
\end{lemma}

Next we  state an excess decay estimate for a homogeneous comparison problem.
\begin{lemma}(see \cite{b13}, lemma 4.1)\label{lemma1.6}
If $ z\in W_{loc}^{1,G}(\Omega)$ is a local weak solution of
\begin{equation}
-\operatorname{div}\left(  a(Dz)\right) =0 \ \ \ \ \ \ in\ \ \ \Omega,
\end{equation}
under the assumptions  \eqref{a(x)1} and \eqref{a(x)3}, then there exist constants $\beta \in (0,1)$ and $c=c(data)$ such that
\begin{equation}
\fint_{B_\rho(x_0)} \vert Dz-(Dz)_{B_\rho(x_0)}\vert \operatorname{d}\!\xi\leq c\left( \frac{\rho}{R}\right) ^\beta \fint_{B_R(x_0)} \vert Dz-(Dz)_{B_R(x_0)}\vert \operatorname{d}\!\xi,
\end{equation}
where $0<\rho\leqslant R, \ B_{2R}(x_0)\subset \Omega.$
\end{lemma}

Our aim is to establish a similar excess decay estimate, incorporating error terms, for solutions to \eqref{u0}. To accomplish this, we utilize a multistep comparison argument to transfer the excess decay estimate from Lemma \ref{lemma1.6} onto solutions of the double obstacle problem.

\begin{lemma}\label{zongjie}
Assume that conditions  \eqref{a(x)1}-\eqref{a(x)3}  are fulfilled, let  $B_{2R}(x_0) \subset \Omega,\psi_1, \psi_2 \in W^{1,G}(\Omega)\cap W^{2,1}(\Omega)$,  $ \ddiv a(x,D\psi_2),  \frac{g(|D\psi_1|)}{|D\psi_1|}|D^{2}\psi_1| \in L^{1}_{loc}(\Omega), $ Let $u\in W^{1,1}(B_R(x_0))$ with $u\geqslant \psi_1$ a.e. be a limit of approximating solutions for $OP(\psi_1, \psi_2; \mu)$ with measure data $\mu\in \mathcal{M}_{b}(B_R(x_0))$ (in the sense of Definition \ref{opdy}).
Then there exists $\beta \in (0,1)$ such that
\begin{eqnarray*}
&&\fint_{B_{\rho}(x_0)}|Du-(Du)_{B_{\rho}(x_0)}|dx\leqslant  c \left( \frac{\rho}{R}\right) ^{\beta}\fint_{B_R(x_0)}|Du-(Du)_{B_R(x_0)}|dx \\
&+& c\left(\frac{R}{\rho} \right) ^n \left[\frac{|\mu|(\overline{B_R(x_0)})}{R^{n-1}} +R\fint_{B_R}|\ddiv a(x,D\psi_2)|dx\right] ^{\frac{1}{i_g}}  \\
&+& c\left(\frac{R}{\rho} \right) ^n\left[R\fint_{B_R(x_0)}\left(\frac{g(|D\psi_1|)}{|D\psi_1|}|D^{2}\psi_1|+1\right)dx\right]^{\frac{1}{i_g}}\\
&+&c\left(\frac{R}{\rho} \right) ^{n}  \omega(R)^{\frac{1}{1+s_g}}\left\lbrace \fint_{B_{R(x_0)}}|Du|dx+ G^{-1}\left[\fint_{B_{R(x_0)}}[G(|D\psi_1|)+G(|\psi_1|)]dx \right]\right\rbrace,
\end{eqnarray*}
where $0<\rho\leqslant R, \ c=c(data)$ and  $\beta$ is as in Lemma \ref{lemma1.6}.
\end{lemma}
\begin{remark}
As we have successfully derived all the requisite comparison estimates for regularized problems at the $L^{1}-$level, our results seamlessly extend to  problems involving measure data.
\end{remark}
\begin{proof}
Without loss of generality we may assume that $x_0=0$ and that $w_1, w_2, w_3,w_4,w_5 \in W^{1,G}(B_R)$ satisfy separately
\begin{equation*}
\left\{\begin{array}{r@{\ \ }c@{\ \ }ll}
&\int_{B_R}&  a(x, Dw_1)\cdot D(v-w_1)dx \geqslant \int_{B_R}  a(x, D\psi_2)\cdot D(v-w_1)dx  \,, \\[0.05cm]
& \mbox{for}&  \forall \ v \in w_1+W_{0}^{1,G}(B_R) \ \mbox{with} \ v\geqslant \psi_1 \  a.e. \ \mbox{in}\ \ B_R \,, \\[0.05cm]
&w_1&\geqslant \psi_1, w_1\geqslant0   \ \ \ \ \ \ a.e. \ \mbox{in}\ \ B_R\,, \\[0.05cm]
&w_1&=u \ \ \ \ \ \ \ \ \ \ \ \ \ \ \ \ \ \mbox{on}\ \ \partial B_R \,,
\end{array}\right.
\end{equation*}

\begin{equation*}
\left\{\begin{array}{r@{\ \ }c@{\ \ }ll}
&\int_{B_R}&  a(x, Dw_2)\cdot D(v-w_2)dx \geqslant0, \,, \\[0.05cm] 
&\mbox{for}& \ \ \forall \ v \in w_2+W_{0}^{1,G}(B_R) \ \mbox{with} \ v\geqslant \psi_1 \  a.e. \ \mbox{in}\ \ B_R \,, \\[0.05cm]
&w_2&\geqslant \psi_1, w_2\geqslant0   \ \ \ \ \ \ a.e. \ \mbox{in}\ \ B_R\,, \\[0.05cm]
&w_2&=w_1 \ \ \ \ \ \ \ \ \ \ \ \ \ \ \ \ \ \mbox{on}\ \ \partial B_R \,,
\end{array}\right.
\end{equation*}

\begin{equation*}
\left\{\begin{array}{r@{\ \ }c@{\ \ }ll}
&\int_{B_\frac{R}{2}}&  \overline{a}_{B_{\frac{R}{2}}}(Dw_3)\cdot D(v-w_3)dx \geqslant0,\  \,, \\[0.05cm]
&\mbox{for}& \ \ \forall \ v \in w_3+W_{0}^{1,G}(B_{\frac{R}{2}}) \ \mbox{with} \ v\geqslant \psi_1 \  a.e. \ \mbox{in}\ \ B_{\frac{R}{2}} \,, \\[0.05cm]
&w_3&\geqslant \psi_1 \ \ \ \ \ \ \ \ \ \ \ \ \ \ \ \ \ a.e. \ \mbox{in}\ \ B_{\frac{R}{2}}\,, \\[0.05cm]
&w_3&=w_2 \ \ \ \ \ \ \ \ \ \ \ \ \ \ \ \ \ \mbox{on}\ \ \partial B_{\frac{R}{2}} \,,
\end{array}\right.
\end{equation*}

\begin{equation*}
\left\{\begin{array}{r@{\ \ }c@{\ \ }ll}
-\operatorname{div}\left(  \overline{a}_{B_{\frac{R}{2}}}(Dw_4)\right)&=&-\operatorname{div}\left( \overline{a}_{B_{\frac{R}{2}}}(D\psi_1)\right) \ \ \ \ \  \ \mbox{in}\ \ B_{\frac{R}{2}} \,, \\[0.05cm]
w_4&=&w_3  \ \  \ \ \ \  \ \ \ \ \ \ \ \ \ \ \ \ \ \ \ \ \ \ \ \  \ \ \mbox{on}\ \ \partial B_{\frac{R}{2}} \,, \\[0.05cm]
\end{array}\right.
\end{equation*}

\begin{equation*}
\left\{\begin{array}{r@{\ \ }c@{\ \ }ll}
-\operatorname{div}\left(  \overline{a}_{B_{\frac{R}{2}}}(Dw_5)\right)&=&0  \ \ \ \ \ \mbox{in}\ \ B_{\frac{R}{2}} \,, \\[0.05cm]
w_5&=&w_4  \ \ \ \mbox{on}\ \ \partial B_{\frac{R}{2}} \,. \\[0.05cm]
\end{array}\right.
\end{equation*}
In order to transfer the double obstacle problems to the single one, we make use of  Corollary \ref{coro} to get  the comparison estimate
\begin{equation}\label{ll2}
\fint_{B_R} |Du-Dw_1| \operatorname{d}\!x \leqslant c\left[\frac{|\mu|(\overline{B_R})}{R^{n-1}}+R\fint_{B_R}|\ddiv a(x,D\psi_2)|dx \right] ^{\frac{1}{i_g}}.
\end{equation}
Using Lemma \ref{w1-w2} to remove the inhomogeneity,
\begin{equation}\label{ll6}
\fint_{B_R} |Dw_1-Dw_2| \operatorname{d}\!x \leqslant c\left[R\fint_{B_R}|\ddiv a(x,D\psi_2)|dx \right] ^{\frac{1}{i_g}}.
\end{equation}
Then we  use  Lemma \ref{diyi} to obtain
\begin{eqnarray}\label{ll5}
\fint_{B_{\frac{R}{2}}}|Dw_2-Dw_3|dx &\leqslant&  c \nonumber \omega(R)^{\frac{1}{1+s_g}}\left\lbrace \fint_{B_{R}}|Dw_2|dx+G^{-1}\left[\fint_{B_{R}}[G(|D\psi_1|)+G(|\psi_1|)]dx \right] \right\rbrace \\  \nonumber
&\leqslant& c \omega(R)^{\frac{1}{1+s_g}}\left\lbrace \fint_{B_{R}}|Du|dx+\left[\frac{|\mu|(\overline{B_R})}{R^{n-1}} +R\fint_{B_R}|\ddiv a(x,D\psi_2)|dx\right] ^{\frac{1}{i_g}}\right\rbrace \\
&+&  c \omega(R)^{\frac{1}{1+s_g}}\left\lbrace G^{-1}\left[\fint_{B_{R}}[G(|D\psi_1|)+G(|\psi_1|)]dx \right] \right\rbrace.
\end{eqnarray}
In the subsequent step, to effectively transition to an obstacle-free problem, we make use of Lemma \ref{zabj} to  to get
\begin{equation}\label{l3}
\fint_{B_{\frac{R}{2}}}|Dw_3-Dw_4|dx\leqslant c\left(R\fint_{B_{\frac{R}{2}}}(|\operatorname{div}\left(  \overline{a}_{B_{\frac{R}{2}}}(D\psi_1)\right)|+1)dx\right)^{\frac{1}{i_g}}.
\end{equation}
We now proceed to simplify the problem by reducing it to a homogeneous equation.  An application of  Lemma \ref{fbjg} with $f=-\operatorname{div}\left(  \overline{a}_{B_{\frac{R}{2}}}(D\psi_1)\right)$ and $g=0$ implies that
\begin{equation}\label{l4}
\fint_{B_{\frac{R}{2}}}|Dw_4-Dw_5|dx\leqslant c\left(R\fint_{B_{\frac{R}{2}}}(|\operatorname{div}\left(  \overline{a}_{B_{\frac{R}{2}}}(D\psi_1)\right)|+1)dx\right)^{\frac{1}{i_g}}.
\end{equation}
Applying Lemma \ref{lemma1.6},  there exists $\beta \in (0,1)$ such that
\begin{eqnarray}\label{ll1}
 \nonumber  \fint_{B_\rho} \vert Dw_5-(Dw_5)_{B_\rho}\vert dx &\leqslant & c\left( \frac{\rho}{R}\right) ^\beta \fint_{B_{\frac{R}{2}}} \vert Dw_5-(Dw_5)_{B_{\frac{R}{2}}}\vert dx \\
&\leqslant & c\left( \frac{\rho}{R}\right) ^\beta \fint_{B_{\frac{R}{2}}}|Dw_5-Du| +\vert Du-(Du)_{B_R}\vert dx.
\end{eqnarray}
Finally combining  \eqref{ll2}-\eqref{ll1} yields
\begin{eqnarray*}
&&\fint_{B_{\rho}}|Du-(Du)_{B_{\rho}}|dx \leqslant \fint_{B_{\rho}}|Du-(Dw_5)_{B_{\rho}}|dx \\
&\leqslant& \fint_{B_{\rho}}(|Du-Dw_1|+|Dw_1-Dw_2|+|Dw_2-Dw_3| \\
&+&|Dw_3-Dw_4|+|Dw_4-Dw_5|+\vert Dw_5-(Dw_5)_{B_\rho}\vert )dx  \\
&\leqslant& c\left(\frac{R}{\rho} \right) ^n \left[ \left[\frac{|\mu|(\overline{B_{\frac{R}{2}})}}{R^{n-1}} +R\fint_{B_R}|\ddiv a(x,D\psi_2)|dx\right] ^{\frac{1}{i_g}}+\left(R\fint_{B_{\frac{R}{2}}}(|\operatorname{div}\left(  \overline{a}_{B_{\frac{R}{2}}}(D\psi_1)\right)|+1)dx\right)^{\frac{1}{i_g}}\right] \\
&+&c\left(\frac{R}{\rho} \right) ^{n}  \omega(R)^{\frac{1}{1+s_g}}\left\lbrace \fint_{B_{R}}|Du|dx+\left[\frac{|\mu|(\overline{B_R})}{R^{n-1}} +R\fint_{B_R}|\ddiv a(x,D\psi_2)|dx\right] ^{\frac{1}{i_g}} \right\rbrace \\
&+&c\left(\frac{R}{\rho} \right) ^{n}  \omega(R)^{\frac{1}{1+s_g}}G^{-1}\left[\fint_{B_{R}}[G(|D\psi_1|)+G(|\psi_1|)]dx \right] \\
&+& c\left( \frac{\rho}{R}\right) ^\beta \left[ \fint_{B_{\frac{R}{2}}}(|Dw_5-Du|+|Du-(Du)_{B_R}|)dx\right]  \\
&\leqslant& c \left( \frac{\rho}{R}\right) ^\beta  \fint_{B_R}|Du-(Du)_{B_R}|dx \\
&+& c\left(\frac{R}{\rho} \right) ^n \left[ \left[\frac{|\mu|(\overline{B_R})}{R^{n-1}} +R\fint_{B_R}|\ddiv a(x,D\psi_2)|dx\right] ^{\frac{1}{i_g}}+\left(R\fint_{B_R}\left(\frac{g(|D\psi_1|)}{|D\psi_1|}|D^{2}\psi_1| +1\right)dx\right)^{\frac{1}{i_g}}\right] \\
&+&c\left(\frac{R}{\rho} \right) ^{n}  \omega(R)^{\frac{1}{1+s_g}}\left\lbrace \fint_{B_{R}}|Du|dx+ G^{-1}\left[\fint_{B_{R}}[G(|D\psi_1|)+G(|\psi_1|)]dx \right] \right\rbrace,
\end{eqnarray*}
where we used the fact $\omega \leqslant2L$ and the condition  \eqref{a(x)1} in the last step.
\end{proof}

\section{The proof of gradient estimates }\label{section5}
The focus of this section lies in deriving pointwise and oscillation estimates for the gradients of solutions utilizing the sharp maximal function estimates. We initiate the analysis by establishing a pointwise estimate of the fractional maximal operator through the utilization of precise iteration methods.
\begin{proof}[Proof of Theorem \ref{Th1}]
\textup{\textbf{Proof of \eqref{1.11}}}
We  define
\begin{equation}
B_i:=B\left( x,\frac{R}{H^i}\right )=B(x,R_i),\ \ \mbox{for}\ \ \ i=0,1,2,...,\ \ \nonumber
\end{equation}
\begin{equation}
A_i:=\fint_{B_i}|Du-(Du)_{B_i}|dx,\ \ k_i:=|(Du)_{B_i}-S|,\ \ \ S\in \mathbb{R}^n.  \nonumber
\end{equation}
We take $H=H(data)>1$ large enough to have
\begin{equation}
c\left(\frac{1}{H} \right) ^\beta\leq \frac{1}{4} \nonumber
\end{equation}
with $\beta$  as in Lemma \ref{lemma1.6} and using Lemma \ref{zongjie}  gives
\begin{eqnarray*}
&&\fint_{B_{i+1}}|Du-(Du)_{B_{i+1}}|\operatorname{d}\! \xi \\ &\leqslant& \frac{1}{4}\fint_{B_{i}}|Du-(Du)_{B_{i}}|\operatorname{d}\!\xi+cH^{n} \left\lbrace  \left[\frac{|\mu|(\overline{B_i})}{R_{i}^{n-1}} \right] ^{\frac{1}{i_g}}+\left[\frac{D\Psi_1(B_i)}{R_{i}^{n-1}} \right] ^{\frac{1}{i_g}}+\left[\frac{D\Psi_2(B_i)}{R_{i}^{n-1}} \right] ^{\frac{1}{i_g}}\right\rbrace  \\
&+&cH^{n}  \omega(R_i)^{\frac{1}{1+s_g}}\left\lbrace \fint_{B_{i}}|Du|d\xi+ G^{-1}\left[\fint_{B_{i}}[G(|D\psi_1|)+G(|\psi_1|)]d\xi \right]\right\rbrace.
\end{eqnarray*}
Our next step involves reducing the value of $R_0$-in a way depending only on $data$ and $\omega(\cdot)$- to get
\begin{equation*}
cH^{n}\omega(R_i)^{\frac{1}{1+s_g}}\leqslant cH^{n}\omega(R_0)^{\frac{1}{1+s_g}}\leqslant \frac{1}{4},
\end{equation*}
which together with  the following estimate
\begin{equation*}
\fint_{B_i}|Du|d\xi \leqslant \fint_{B_i}|Du-(Du)_{B_i}|d\xi+k_i+|S|,
\end{equation*}
we derive that
\begin{eqnarray}\label{1.19}
\nonumber A_{i+1}&\leqslant& \frac{1}{2}A_i+c\left\lbrace \left[\frac{|\mu|(\overline{B_i})}{R_i^{n-1}} \right] ^{\frac{1}{i_g}} +\left[\frac{D\Psi_1(B_i)}{R_i^{n-1}} \right] ^{\frac{1}{i_g}}+\left[\frac{D\Psi_2(B_i)}{R_{i}^{n-1}} \right] ^{\frac{1}{i_g}}\right\rbrace  \\
&+&c[\omega(R_i)]^{\frac{1}{1+s_g}}\left\lbrace k_i+|S|+ G^{-1}\left[\fint_{B_{i}}[G(|D\psi_1|)+G(|\psi_1|)]d\xi \right]\right\rbrace
\end{eqnarray}
whenever $ i\geqslant0 $.  On the other hand, we calculate
\begin{eqnarray*}
|k_{i+1}-k_i|&\leq&|(Du)_{B_{i+1}}-(Du)_{B_i}| \\
&\leq& \fint_{B_{i+1}}|Du-(Du)_{B_{i}}|\operatorname{d}\!\xi\\
&\leq& H^n \fint_{B_i}|Du-(Du)_{B_{i}}|\operatorname{d}\!\xi=H^nA_i,
\end{eqnarray*}
from which we see that  for $ m\in \mathbb{N} $,
\begin{equation}\label{1.20}
k_{m+1}=\sum_{i=0}^m(k_{i+1}-k_i)+k_0\leq H^n\sum_{i=0}^mA_i+k_0.
\end{equation}
In this step, we perform a summation of \eqref{1.19} across the range of indices $ i \in \left\lbrace  0,...,m-1\right\rbrace$, leading us to derive the subsequent inequality
\begin{eqnarray*}
\sum_{i=1}^mA_i&\leq& \frac{1}{2}\sum_{i=0}^{m-1}A_i+c\sum_{i=0}^{m-1}\left\lbrace \left[\frac{|\mu|(\overline{B_i})}{R_i^{n-1}} \right] ^{\frac{1}{i_g}} +\left[\frac{D\Psi_1(B_i)}{R_i^{n-1}} \right] ^{\frac{1}{i_g}}+\left[\frac{D\Psi_2(B_i)}{R_{i}^{n-1}} \right] ^{\frac{1}{i_g}}\right\rbrace  \\
&+&c\sum_{i=0}^{m-1}[\omega(R_i)]^{\frac{1}{1+s_g}}\left\lbrace k_i+|S|+ G^{-1}\left[\fint_{B_{i}}[G(|D\psi_1|)+G(|\psi_1|)]d\xi \right] \right\rbrace.
\end{eqnarray*}
Consequently,
\begin{eqnarray*}
\sum_{i=1}^mA_i&\leq& A_0+2c\sum_{i=0}^{m-1}\left\lbrace \left[\frac{|\mu|(\overline{B_i})}{R_i^{n-1}} \right] ^{\frac{1}{i_g}} +\left[\frac{D\Psi_1(B_i)}{R_i^{n-1}} \right] ^{\frac{1}{i_g}}+\left[\frac{D\Psi_2(B_i)}{R_{i}^{n-1}} \right] ^{\frac{1}{i_g}}\right\rbrace  \\
&+&c\sum_{i=0}^{m-1}[\omega(R_i)]^{\frac{1}{1+s_g}}\left\lbrace k_i+|S|+ G^{-1}\left[\fint_{B_{i}}[G(|D\psi_1|)+G(|\psi_1|)]d\xi \right] \right\rbrace.
\end{eqnarray*}
For every integer $m\geqslant1$ we employ \eqref{1.20} to gain
\begin{eqnarray}\label{2.3}
\nonumber k_{m+1}&\leq& cA_0+ck_0+c\sum_{i=0}^{m-1}\left\lbrace \left[\frac{|\mu|(\overline{B_i})}{R_i^{n-1}} \right] ^{\frac{1}{i_g}} +\left[\frac{D\Psi_1(B_i)}{R_i^{n-1}} \right] ^{\frac{1}{i_g}}+\left[\frac{D\Psi_2(B_i)}{R_{i}^{n-1}} \right] ^{\frac{1}{i_g}}\right\rbrace  \\
&+&c\sum_{i=0}^{m-1}[\omega(R_i)]^{\frac{1}{1+s_g}}\left\lbrace k_i+|S|+ G^{-1}\left[\fint_{B_{i}}[G(|D\psi_1|)+G(|\psi_1|)]d\xi \right] \right\rbrace.
\end{eqnarray}
Taking into consideration the explicit definition of $A_0$, we obtain
\begin{eqnarray}\label{a0guji}
\nonumber &&k_{m+1} \\
&\leq& c\fint_{B_R}|Du-(Du)_{B_R}|+|Du-S|\operatorname{d}\!\xi \\
&+&c\sum_{i=0}^{m-1}\left\lbrace \left[\frac{|\mu|(\overline{B_i})}{R_i^{n-1}} \right] ^{\frac{1}{i_g}} +\left[\frac{D\Psi_1(B_i)}{R_i^{n-1}} \right] ^{\frac{1}{i_g}}+\left[\frac{D\Psi_2(B_i)}{R_{i}^{n-1}} \right] ^{\frac{1}{i_g}}\right\rbrace  \\
&+&c\sum_{i=0}^{m-1}[\omega(R_i)]^{\frac{1}{1+s_g}}\left\lbrace k_i+|S|+ G^{-1}\left[\fint_{B_{i}}[G(|D\psi_1|)+G(|\psi_1|)]d\xi \right] \right\rbrace
\end{eqnarray}
for every $m\geqslant0.$
By choosing $S=0$  in the inequality mentioned earlier and subsequently multiplying both sides by $R_{m+1}^{1-\alpha}$
,  while considering the given restrictions
 $\alpha\in[0,1]$ and $R_{m+1}\leq R_i$ for $0\leqslant i \leqslant m+1$,  we obtain
\begin{eqnarray*}
&&R_{m+1}^{1-\alpha}k_{m+1} \\
&\leq& cR^{1-\alpha} \fint_{B_R}|Du|\operatorname{d}\!\xi+c\sum_{i=0}^{m-1}R_i^{1-\alpha}\left\lbrace \left[\frac{|\mu|(\overline{B_i})}{R_i^{n-1}} \right] ^{\frac{1}{i_g}} +\left[\frac{D\Psi_1(B_i)}{R_i^{n-1}} \right] ^{\frac{1}{i_g}}+\left[\frac{D\Psi_2(B_i)}{R_{i}^{n-1}} \right] ^{\frac{1}{i_g}}\right\rbrace  \\
&+&c\sum_{i=0}^{m-1}R_i^{1-\alpha}[\omega(R_i)]^{\frac{1}{1+s_g}}\left\lbrace k_i+ G^{-1}\left[\fint_{B_{i}}[G(|D\psi_1|)+G(|\psi_1|)]d\xi \right] \right\rbrace
\end{eqnarray*}
and therefore
\begin{eqnarray}\label{1.22}
\nonumber &&R_{m+1}^{1-\alpha}k_{m+1}\\
&\leq& c\sum_{i=0}^{m}\left\lbrace \left[\frac{|\mu|(\overline{B_i})}{R_i^{n-1-i_g(1-\alpha)}} \right] ^{\frac{1}{i_g}}+\left[\frac{D\Psi_1(B_i)}{R_i^{n-1-i_g(1-\alpha)}} \right] ^{\frac{1}{i_g}}+\left[\frac{D\Psi_2(B_i)}{R_i^{n-1-i_g(1-\alpha)}} \right] ^{\frac{1}{i_g}}\right\rbrace \\
&+&cR^{1-\alpha} \fint_{B_R}|Du|\operatorname{d}\!\xi+c\sum_{i=0}^{m}R_i^{1-\alpha}[\omega(R_i)]^{\frac{1}{1+s_g}}\left\lbrace k_i+ G^{-1}\left[\fint_{B_{i}}[G(|D\psi_1|)+G(|\psi_1|)]d\xi \right] \right\rbrace.
\end{eqnarray}
By employing the definition of Wolff potential, we proceed to estimate the final term on the right-hand side in equation \eqref{1.22}, yielding
\begin{eqnarray}\label{1.23}
\nonumber&&\sum_{i=0}^{\infty}\left[\frac{|\mu|(\overline{B_i})}{R_i^{n-1-i_g(1-\alpha)}} \right] ^{\frac{1}{i_g}} \\
&\leq& \frac{c}{\log2}\int_R^{2R}\left[\frac{|\mu|(\overline{B_\rho})}{\rho^{n-1-i_g(1-\alpha)}} \right]  ^{\frac{1}{i_g}}\frac{\operatorname{d}\!\rho}{\rho}\nonumber
+\sum_{i=0}^{\infty}\frac{c}{\log H}\int_{R_{i+1}}^{R_i}\left[\frac{|\mu|(\overline{B_\rho})}{\rho^{n-1-i_g(1-\alpha)}} \right]  ^{\frac{1}{i_g}}\frac{\operatorname{d}\!\rho}{\rho}\\
&\leq& cW_{1-\alpha+\frac{\alpha}{i_g+1},i_g+1}^{\mu}(x,2R).
\end{eqnarray}
At the last step, we utilize the assumption $|\mu|(B_{2R})<+\infty$ to deduce that the condition  $|\mu|(\partial B_{\rho})>0$ can hold for only a countable number of radii
$\rho \in (R,2R)$.
Likewise, we have
\begin{equation*}
\sum_{i=0}^{\infty}\left[\frac{D\Psi_1(B_i)}{R_i^{n-1-i_g(1-\alpha)}} \right] ^{\frac{1}{i_g}} \leqslant cW_{1-\alpha+\frac{\alpha}{i_g+1},i_g+1}^{[\psi_1]}(x,2R),
\end{equation*}
\begin{equation*}
\sum_{i=0}^{\infty}\left[\frac{D\Psi_2(B_i)}{R_i^{n-1-i_g(1-\alpha)}} \right] ^{\frac{1}{i_g}} \leqslant cW_{1-\alpha+\frac{\alpha}{i_g+1},i_g+1}^{[\psi_2]}(x,2R),
\end{equation*}
\begin{equation*}
\sum_{i=0}^{\infty}\left[\omega(R_i)\right] ^{\frac{1}{1+s_g}}\leqslant c\int_{0}^{2R}[\omega(\rho)]^{\frac{1}{1+s_g}}\frac{d\rho}{\rho},
\end{equation*}
\begin{eqnarray*}
\sum_{i=0}^{\infty}R_{i}^{1-\alpha}\left[\omega(R_i)\right] ^{\frac{1}{1+s_g}}G^{-1}\left[\fint_{B_{i}}[G(|D\psi|)+G(|\psi|)]d\xi \right] \\
\leqslant c\int_{0}^{2R}[\omega(\rho)]^{\frac{1}{1+s_g}}G^{-1}\left[\fint_{B_{\rho}}[G(|D\psi_1|)+G(|\psi_1|)]d\xi \right]\frac{d\rho}{\rho^{\alpha}}.
\end{eqnarray*}
Consequently,
\begin{equation}\label{2.7}
R_{m+1}^{1-\alpha}k_{m+1}\leqslant cM+c\sum_{i=0}^{m}R_{i}^{1-\alpha}k_{i}\left[\omega(R_i)\right] ^{\frac{1}{1+s_g}},
\end{equation}
where
\begin{eqnarray*}
M:&=&c\left[ W_{1-\alpha+\frac{\alpha}{i_g+1},i_g+1}^{\mu}(x,2R)+W_{1-\alpha+\frac{\alpha}{i_g+1},i_g+1}^{[\psi_1]}(x,2R)+W_{1-\alpha+\frac{\alpha}{i_g+1},i_g+1}^{[\psi_2]}(x,2R)\right]  \\
&+&cR^{1-\alpha}\fint_{B_R}|Du|d\xi +c\int_{0}^{2R}[\omega(\rho)]^{\frac{1}{1+s_g}}G^{-1}\left[\fint_{B_{\rho}}[G(|D\psi_1|)+G(|\psi_1|)]d\xi \right]\frac{d\rho}{\rho^{\alpha}}.
\end{eqnarray*}
Using the same arguments used in \cite{xiong2}, we  prove by induction that
\begin{equation}\label{mguji}
R_{m+1}^{1-\alpha}k_{m+1}\leqslant (c+c^{*})M
\end{equation}
holds for every $m\geqslant 0$, some positive constants $c,c^{*}>1$.

Now  we define
\begin{equation*}
C_m:=R_m^{1-\alpha}A_m=R_m^{1-\alpha}\fint_{B_m}|Du-(Du)_{B_m}|d\xi
\end{equation*}
\begin{equation*}
h_m:=\fint_{B_m}|Du|d\xi.
\end{equation*}
Then it's  easy to obtain
\begin{eqnarray*}
R_m^{1-\alpha}h_m&=&R_m^{1-\alpha}\fint_{B_m}|Du|d\xi \\
&\leq& R_m^{1-\alpha}\fint_{B_m}|Du-(Du)_{B_m}|+|(Du)_{B_m}|d\xi  \\
&=&R_m^{1-\alpha}k_m+C_m \\
&\leq& CM+C_m
\end{eqnarray*}
with $M$ as in \eqref{2.7}. Consequently, our focus narrows down to finding an upper bound for $C_m$. By employing  \eqref{1.23} and recalling the definition of
$M$ in  \eqref{2.7}, we obtain
\begin{eqnarray*}
\left[\frac{|\mu|(\overline{B_i})}{R_i^{n-1}} \right] ^{\frac{1}{i_g}}
&\leq& CR_i^{\alpha-1}W_{1-\alpha+\frac{\alpha}{i_g+1},i_g+1}^{\mu}(x,2R) \\
&\leq&  CR_i^{\alpha-1}M.
\end{eqnarray*}
Likewise, we have
\begin{equation*}
\left[\frac{D\Psi_1(B_i)}{R_{i}^{n-1}} \right] ^{\frac{1}{i_g}}\leqslant cR_{i}^{\alpha-1}M,
\end{equation*}
\begin{equation*}
\left[\frac{D\Psi_2(B_i)}{R_{i}^{n-1}} \right] ^{\frac{1}{i_g}}\leqslant cR_{i}^{\alpha-1}M,
\end{equation*}
\begin{equation*}
\left[\omega(R_i) \right] ^{\frac{1}{1+s_g}} G^{-1}\left[\fint_{B_{i}}[G(|D\psi_1|)+G(|\psi_1|)]d\xi \right] \leqslant cR_{i}^{\alpha-1}M.
\end{equation*}
Consequently, with reference to  \eqref{1.19}, we establish
\begin{equation*}
A_{m+1}\leq \frac{1}{2}A_m+cR_m^{\alpha-1}M +ck_m.
\end{equation*}
Subsequently, we employ  $\eqref{mguji}$ to obtain
\begin{equation*}
A_{m+1}\leq \frac{1}{2}A_m+cR_m^{\alpha-1}M.
\end{equation*}
then multiply both sides by $R_{m+1}^{1-\alpha}$, we get
\begin{eqnarray}\label{1.25}
C_{m+1}&\leq& \frac{1}{2}\left( \frac{R_{m+1}}{R_m}\right) ^{1-\alpha}C_m \nonumber
+C\left( \frac{R_{m+1}}{R_m}\right) ^{1-\alpha}M \\
&\leq& \frac{1}{2}C_m+C_1M.
\end{eqnarray}
Then applying the same arguments used in \cite{xiong2}, we  prove by induction that
\begin{equation}\label{2.9}
C_m\leq 2C_1M
\end{equation}
holds whenever $m\geqslant0$.
Let us now consider the case where
$r\leqslant R$ and select the integer
$i\geqslant0$ satisfying $R_{i+1}\leqslant r\leqslant R_i$, then we have
\begin{eqnarray*}
r^{1-\alpha}\fint_{B_r}|Du|\operatorname{d}\!\xi
&\leq& \left(\frac{R_i}{R_{i+1}} \right) ^nR_i^{1-\alpha}\fint_{B_i}|Du|\operatorname{d}\!\xi \\
&\leq& CH^nR_i^{1-\alpha}h_i \\
&\leq& CM.
\end{eqnarray*}
With the definition of  $M$ and the restricted maximal operator in mind, we obtain in turn
\begin{eqnarray*}
&&M_{1-\alpha,R}(|Du|)(x)\\
&\leqslant &c\left[ W_{1-\alpha+\frac{\alpha}{i_g+1},i_g+1}^{\mu}(x,2R)+W_{1-\alpha+\frac{\alpha}{i_g+1},i_g+1}^{[\psi_1]}(x,2R)+W_{1-\alpha+\frac{\alpha}{i_g+1},i_g+1}^{[\psi_2]}(x,2R)\right]  \\
&+&cR^{1-\alpha}\fint_{B_R}|Du|d\xi+c\int_{0}^{2R}[\omega(\rho)]^{\frac{1}{1+s_g}}G^{-1}\left[\fint_{B_{\rho}}[G(|D\psi_1|)+G(|\psi_1|)]d\xi \right]\frac{d\rho}{\rho^{\alpha}}.
\end{eqnarray*}
Moreover, because of
\begin{equation*}
M_{\alpha,R}^{\#}(u)(x)\leqslant M_{1-\alpha,R}(|Du|)(x),
\end{equation*}
then we have
\begin{eqnarray*}
&&M_{\alpha,R}^{\#}(u)(x)+M_{1-\alpha,R}(Du)(x)\\
&\leqslant& c\left[ W_{1-\alpha+\frac{\alpha}{i_g+1},i_g+1}^{\mu}(x,2R)+W_{1-\alpha+\frac{\alpha}{i_g+1},i_g+1}^{[\psi_1]}(x,2R)+W_{1-\alpha+\frac{\alpha}{i_g+1},i_g+1}^{[\psi_2]}(x,2R)\right]  \\
&+&cR^{1-\alpha}\fint_{B_R}|Du|d\xi+c\int_{0}^{2R}[\omega(\rho)]^{\frac{1}{1+s_g}}G^{-1}\left[\fint_{B_{\rho}}[G(|D\psi_1|)+G(|\psi_1|)]d\xi \right]\frac{d\rho}{\rho^{\alpha}}.
\end{eqnarray*}
\textup{\textbf{Proof of \eqref{1.122}}}
We define
\begin{equation}\label{a}
\widehat{A}_i:=R_i^{-\alpha}\fint_{B_i}|Du-(Du)_{B_i}|\operatorname{d}\!\xi
\end{equation}
Using Lemma \ref{zongjie} (multiply both sides by $R_{i+1}^{-\alpha}$) we obtain
\begin{eqnarray*}
&&\widehat{A}_{i+1} \\
&\leqslant &c\left( \frac{R_{i+1}}{R_{i}}\right) ^{\beta-\alpha}\widehat{A}_i+c \left( \frac{R_{i}}{R_{i+1}}\right) ^{n+\alpha}\left\lbrace \left[ \frac{|\mu|(\overline{B_i})}{R_i^{n-1+\alpha i_g}}\right] ^{\frac{1}{i_g}}+ \left[ \frac{D\Psi_1(B_i)}{R_i^{n-1+\alpha i_g}}\right] ^{\frac{1}{i_g}}+\left[ \frac{D\Psi_2(B_i)}{R_i^{n-1+\alpha i_g}}\right] ^{\frac{1}{i_g}}\right\rbrace \\
&+&c\left( \frac{R_{i}}{R_{i+1}}\right) ^{n+\alpha}\frac{1}{R_i^{\alpha}} \omega(R_i)^{\frac{1}{1+s_g}}\left\lbrace \fint_{B_{i}}|Du|d\xi+ G^{-1}\left[\fint_{B_{i}}[G(|D\psi_1|)+G(|\psi_1|)]d\xi \right] \right\rbrace.
\end{eqnarray*}
We choose a value of  $H=H(n,i_a,s_a,\widehat{\alpha},\beta)>1$ that is sufficiently large to obtain
\begin{equation*}
c\left( \frac{R_{i+1}}{R_{i}}\right) ^{\beta-\alpha}
=c\left( \frac{1}{H}\right)^{\beta-\alpha}
\leqslant c\left( \frac{1}{H}\right)^{\beta-\widehat{\alpha} }
\leqslant\frac{1}{2}
\end{equation*}
and
\begin{equation*}
 \frac{|\mu|(\overline{B_i})}{R_i^{n-1+\alpha i_g}}\leqslant H^{n-1+\alpha i_g}\frac{|\mu|(B_{i-1})}{R_{i-1}^{n-1+\alpha i_g}}.
\end{equation*}
Due to assumption \eqref{wtiaojian}, we have
$$\frac{[\omega(R_i)]^{\frac{1}{1+s_g}}}{R_{i}^{\alpha}}\leqslant \frac{[\omega(R_i)]^{\frac{1}{1+s_g}}}{R_{i}^{\widehat{\alpha}}}\leqslant c_0,$$
then according to  the definition of restricted maximal operator, we derive that
\begin{eqnarray*}
&&\widehat{A}_{i+1}\\
&\leqslant& \frac{1}{2}\widehat{A}_{i}+c\left\lbrace \left[ M_{1-\alpha i_g,R}(\mu)(x)\right] ^{\frac{1}{i_g}}+\left[ \overline{M}_{1-\alpha i_g,R}(\psi_1)(x)\right] ^{\frac{1}{i_g}}+\left[ \overline{M}_{1-\alpha i_g,R}(\psi_2)(x)\right] ^{\frac{1}{i_g}}+\fint_{B_{i}}|Du|d\xi \right\rbrace \\
&+&c\frac{1}{R_i^{\alpha}} \omega(R_i)^{\frac{1}{1+s_g}}\left\lbrace  G^{-1}\left[\fint_{B_{i}}[G(|D\psi_1|)+G(|\psi_1|)]d\xi\right] \right\rbrace
\end{eqnarray*}
for every $i\geqslant0$.
Moreover, from the case $\alpha=1$ of inequality \eqref{1.11}, we have
\begin{eqnarray}\label{mwanguji}
\nonumber &&\fint_{B_i}|Du|d\xi \\ \non
&\leqslant& c\left[ \fint_{B_R}|Du|d\xi+W_{\frac{1}{i_g+1},i_g+1}^{\mu}(x,2R)+W_{ \frac{1}{i_g+1},i_g+1}^{[\psi_1]}(x,2R)+W_{ \frac{1}{i_g+1},i_g+1}^{[\psi_2]}(x,2R)\right]  \\ \nonumber
&+&c\int_{0}^{2R}[\omega(\rho)]^{\frac{1}{1+s_g}}G^{-1}\left[\fint_{B_{\rho}}[G(|D\psi_1|)+G(|\psi_1|)]d\xi \right]\frac{d\rho}{\rho} \\
&:=&cM^{*}.
\end{eqnarray}
Therefore, by combining with the two previous inequalities, we obtain
\begin{eqnarray*}
&&\widehat{A}_{i+1}\\
&\leqslant& \frac{1}{2}\widehat{A}_{i}+c\left\lbrace \left[ M_{1-\alpha i_g,R}(\mu)(x)\right] ^{\frac{1}{i_g}}+\left[ \overline{M}_{1-\alpha i_g,R}(\psi_1)(x)\right] ^{\frac{1}{i_g}}+\left[ \overline{M}_{1-\alpha i_g,R}(\psi_2)(x)\right] ^{\frac{1}{i_g}}+M^{*}\right\rbrace \\
&+&c\frac{1}{R_i^{\alpha}} \omega(R_i)^{\frac{1}{1+s_g}}\left\lbrace  G^{-1}\left[\fint_{B_{i}}[G(|D\psi_1|)+G(|\psi_1|)]d\xi \right] \right\rbrace.
\end{eqnarray*}
Iterating the previous relation, we conclude
\begin{eqnarray*}
&& \widehat{A}_{i}\leqslant 2^{-i} \widehat{A}_{0}\\
&+&c\sum_{j=0}^{i-1}2^{-j} \left\lbrace \left[ M_{1-\alpha i_g,R}(\mu)(x)\right] ^{\frac{1}{i_g}}+\left[ \overline{M}_{1-\alpha i_g,R}(\psi_1)(x)\right] ^{\frac{1}{i_g}}+\left[ \overline{M}_{1-\alpha i_g,R}(\psi_2)(x)\right] ^{\frac{1}{i_g}}+M^{*}\right\rbrace  \\
 &+& c\sum_{j=0}^{i-1}2^{-j}  \frac{1}{R_j^{\alpha}} \omega(R_j)^{\frac{1}{1+s_g}}\left\lbrace   G^{-1}\left[\fint_{B_{j}}[G(|D\psi_1|)+G(|\psi_1|)]d\xi \right] \right\rbrace
\end{eqnarray*}
holds for every $i\geqslant1$.
Subsequently, analogous to  \eqref{1.23},  we derive
\begin{eqnarray*}
&&\sum_{j=0}^{i-1}2^{-j}  \frac{1}{R_j^{\alpha}} \omega(R_j)^{\frac{1}{1+s_g}}\left\lbrace   G^{-1}\left[\fint_{B_{j}}[G(|D\psi_1|)+G(|\psi_1|)]d\xi \right] \right\rbrace  \\
&\leqslant &c\int_{0}^{2R}[\omega(\rho)]^{\frac{1}{1+s_g}}G^{-1}\left[\fint_{B_{\rho}}[G(|D\psi_1|)+G(|\psi_1|)]d\xi \right]\frac{d\rho}{\rho^{1+\alpha}}.
\end{eqnarray*}
By recalling  \eqref{a}  and  \eqref{1.8}, it is straightforward to deduce that
\begin{eqnarray*}
&& \sup_{i\geqslant0}\widehat{A}_{i}\\
&\leqslant& c\left\lbrace R^{-\alpha}\fint_{B_R}|Du|\operatorname{d}\!\xi+ \left[  M_{1-\alpha i_g,R}(\mu)(x)\right] ^{\frac{1}{i_g}}+\left[  \overline{M}_{1-\alpha i_g,R}(\psi_1)(x)\right] ^{\frac{1}{i_g}}
+\left[  \overline{M}_{1-\alpha i_g,R}(\psi_2)(x)\right] ^{\frac{1}{i_g}}\right\rbrace  \\
 &+& c\left\lbrace M^{*}+\int_{0}^{2R}[\omega(\rho)]^{\frac{1}{1+s_g}}G^{-1}\left[\fint_{B_{\rho}}[G(|D\psi_1|)+G(|\psi_1|)]d\xi \right]\frac{d\rho}{\rho^{1+\alpha}}
\right\rbrace .
\end{eqnarray*}
For every $ \rho\in (0,R]$, let $ i\in \mathbb{N}$ be  such that $ R_{i+1}<\rho\leqslant R_{i} $,  then we gain
\begin{eqnarray*}
\rho^{-\alpha}\fint_{B_{\rho}}\vert Du-(Du)_{B_{\rho}}\vert\operatorname{d}\!\xi &\leqslant & c\frac{R_{i}^n}{\rho^n}R_{i+1}^{-\alpha}\fint_{B_i}\vert Du-(Du)_{B_i}\vert\operatorname{d}\!\xi  \\
&\leqslant & c\sup_{i\geqslant0}\widehat{A}_{i},
\end{eqnarray*}
from which we derive that
\begin{eqnarray*}
&&M_{\alpha,R}^{\#}(Du)(x)\\
&\leqslant& c\left\lbrace R^{-\alpha}\fint_{B_R}|Du|\operatorname{d}\!\xi+ \left[  M_{1-\alpha i_g,R}(\mu)(x)\right] ^{\frac{1}{i_g}}+\left[  \overline{M}_{1-\alpha i_g,R}(\psi_1)(x)\right] ^{\frac{1}{i_g}}+\left[  \overline{M}_{1-\alpha i_g,R}(\psi_2)(x)\right] ^{\frac{1}{i_g}}\right\rbrace  \\
&+&c\left\lbrace W_{\frac{1}{i_g+1},i_g+1}^{\mu}(x,2R)+W_{ \frac{1}{i_g+1},i_g+1}^{[\psi_1]}(x,2R)+W_{ \frac{1}{i_g+1},i_g+1}^{[\psi_2]}(x,2R)\right\rbrace  \\
&+&c\int_{0}^{2R}[\omega(\rho)]^{\frac{1}{1+s_g}}G^{-1}\left[\fint_{B_{\rho}}[G(|D\psi_1|)+G(|\psi_1|)]d\xi \right]\frac{d\rho}{\rho^{1+\alpha}}.
\end{eqnarray*}
 This completes the proof of Theorem \ref{Th1}.
\end{proof}
\begin{proof}[Proof of Theorem \ref{Th2}]
At first we give the proof of the estimate \eqref{du}.

We choose $S=0$ and make use  of \eqref{a0guji}, we conclude
\begin{eqnarray*}
\nonumber k_{m+1}&\leq& c\fint_{B_R(x_0)}|Du-(Du)_{B_R(x_0)}|+|Du|\operatorname{d}x \\
&+& c\sum_{i=0}^{m-1}\left\lbrace \left[\frac{|\mu|(\overline{B_{R_i}}(x_0))}{R_i^{n-1}} \right] ^{\frac{1}{i_g}} +\left[\frac{D\Psi_1(B_{R_i}(x_0))}{R_i^{n-1}} \right] ^{\frac{1}{i_g}}+\left[\frac{D\Psi_2(B_{R_i}(x_0))}{R_i^{n-1}} \right] ^{\frac{1}{i_g}}\right\rbrace  \\
&+&c\sum_{i=0}^{m-1}[\omega(B_i)]^{\frac{1}{1+s_g}}\left\lbrace \fint_{B_{R_i}(x_0)}|Du|dx+ G^{-1}\left[\fint_{B_{R_i}(x_0)}[G(|D\psi_1|)+G(|\psi_1|)]dx \right]\right\rbrace.
\end{eqnarray*}
On the other hand, we observe
\begin{equation*}
\sum_{i=0}^{+\infty}\left[\frac{|\mu|(\overline{B_{R_i}}(x_0))}{R_i^{n-1}} \right] ^{\frac{1}{i_g}} \leqslant cW_{\frac{1}{i_g+1},i_g+1}^{\mu}(x_0,2R),
\end{equation*}
\begin{equation*}
\sum_{i=0}^{+\infty}\left[\frac{D\Psi_1(B_{R_i}(x_0))}{R_i^{n-1}} \right] ^{\frac{1}{i_g}} \leqslant cW_{\frac{1}{i_g+1},i_g+1}^{[\psi_1]}(x_0,2R),
\end{equation*}
\begin{equation*}
\sum_{i=0}^{+\infty}\left[\frac{D\Psi_2(B_{R_i}(x_0))}{R_i^{n-1}} \right] ^{\frac{1}{i_g}} \leqslant cW_{\frac{1}{i_g+1},i_g+1}^{[\psi_2]}(x_0,2R),
\end{equation*}
\begin{equation*}
\sum_{i=0}^{+\infty}\left[\omega(R_i)\right] ^{\frac{1}{1+s_g}}\leqslant c\int_{0}^{2R}[\omega(\rho)]^{\frac{1}{1+s_g}}\frac{d\rho}{\rho}\leqslant c,
\end{equation*}
\begin{eqnarray*}
\sum_{i=0}^{+\infty}\left[\omega(R_i)\right] ^{\frac{1}{1+s_g}}G^{-1}\left[\fint_{B_{R_i}(x_0)}[G(|D\psi_1|)+G(|\psi_1|)]dx \right] \\
\leqslant c\int_{0}^{2R}[\omega(\rho)]^{\frac{1}{1+s_g}}G^{-1}\left[\fint_{B_{\rho}(x_0)} [G(|D\psi_1|)+G(|\psi_1|)]dx\right]\frac{d\rho}{\rho}.
\end{eqnarray*}
Combining \eqref{mwanguji} with above estimates, it follows that
\begin{eqnarray*}
&&|Du(x_0)|=\lim_{m\rightarrow\infty}k_{m+1} \\
&\leqslant& c\left[ \fint_{B_R(x_0)}|Du|dx+W_{\frac{1}{i_g+1},i_g+1}^{\mu}(x_0,2R)+W_{ \frac{1}{i_g+1},i_g+1}^{[\psi_1]}(x_0,2R)+W_{ \frac{1}{i_g+1},i_g+1}^{[\psi_2]}(x_0,2R)\right] \\
&+&c\int_{0}^{2R}[\omega(\rho)]^{\frac{1}{1+s_g}}G^{-1}\left[\fint_{B_{\rho}(x_0)}[G(|D\psi_1|)+G(|\psi_1|)]dx \right]\frac{d\rho}{\rho}.
\end{eqnarray*}
Next we prove the estimate \eqref{du-du}.
For every $ x,y\in B_\frac{R}{4}(x_0)$,  we define
\begin{equation*}
r_i:=\frac{r}{H^i}, \ \ \ \ r\leqslant \frac{R}{2}, \ \ \ \ k_i=|(Du)_{B_{r_i}(x)}-S|, \ \ \ \ \overline{k_i}=|(Du)_{B_{r_i}(y)}-S|.
\end{equation*}
Taking use  of   \eqref{a0guji} again, we have
\begin{eqnarray*}
k_{m+1}&\leqslant & c\fint_{B_r(x)}(\vert Du-(Du)_{B_r(x)}\vert+\vert Du-S\vert)\operatorname{d}\!\xi \\
&+&c r^{\alpha}\sum_{i=0}^{m-1}\left\lbrace \left[ \frac{|\mu|(\overline{B_{r_i}}(x))}{r_i^{n-1+\alpha i_g}}\right] ^\frac{1}{i_g}+\left[ \frac{D\Psi_1(B_{r_i}(x))}{r_i^{n-1+\alpha i_g}}\right] ^\frac{1}{i_g}+\left[ \frac{D\Psi_2(B_{r_i}(x))}{r_i^{n-1+\alpha i_g}}\right] ^\frac{1}{i_g}\right\rbrace \\
&+&cr^{\alpha}\sum_{i=0}^{m-1}\frac{1}{r_{i}^{\alpha}}[\omega(r_i)]^{\frac{1}{1+s_g}}\left\lbrace \fint_{B_{r_i}(x)}|Du|d\xi+ G^{-1}\left[\fint_{B_{r_i}(x)}[G(|D\psi_1|)+G(|\psi_1|)]d\xi \right] \right\rbrace.
\end{eqnarray*}
Moreover,
\begin{equation*}
\sum_{i=0}^{+\infty}\left[\frac{|\mu|(\overline{B_{r_i}(x)})}{r_i^{n-1+\alpha i_g}} \right] ^{\frac{1}{i_g}}
\leqslant cW_{\frac{1-\alpha i_{g}}{1+i_g},i_g+1}^{\mu}(x,2r),
\end{equation*}
\begin{equation*}
\sum_{i=0}^{+\infty}\left[\frac{D\Psi_1(B_{r_i}(x))}{r_i^{n-1+\alpha i_g}} \right] ^{\frac{1}{i_g}} \leqslant cW_{\frac{1-\alpha i_g}{1+i_g},i_g+1}^{[\psi_1]}(x,2r),
\end{equation*}
\begin{equation*}
\sum_{i=0}^{+\infty}\left[\frac{D\Psi_2(B_{r_i}(x))}{r_i^{n-1+\alpha i_g}} \right] ^{\frac{1}{i_g}} \leqslant cW_{\frac{1-\alpha i_g}{1+i_g},i_g+1}^{[\psi_2]}(x,2r),
\end{equation*}
\begin{equation*}
\sum_{i=0}^{+\infty}\frac{1}{r_{i}^{\alpha}}\left[\omega(r_i)\right] ^{\frac{1}{1+s_g}}\leqslant c\int_{0}^{2r}[\omega(\rho)]^{\frac{1}{1+s_g}}\frac{d\rho}{\rho^{1+\alpha}}\leqslant c,
\end{equation*}
\begin{eqnarray*}
\sum_{i=0}^{+\infty}\frac{1}{r_{i}^{\alpha}}\left[\omega(R_i)\right] ^{\frac{1}{1+s_g}}G^{-1}\left[\fint_{B_{r_i}(x_0)}[G(|D\psi_1|)+G(|\psi_1|)]d\xi \right] \\
\leqslant c\int_{0}^{2r}[\omega(\rho)]^{\frac{1}{1+s_g}}G^{-1}\left[\fint_{B_{\rho}(x)}[G(|D\psi_1|)+G(|\psi_1|)]d\xi \right]\frac{d\rho}{\rho^{1+\alpha}}.
\end{eqnarray*}
Combining \eqref{mwanguji} with the previous estimates,  we obtain
\begin{eqnarray*}
k_{m+1} &\leqslant& c \fint_{B_r(x)}(\vert Du-(Du)_{B_r(x)}\vert+\vert Du-S\vert)\operatorname{d}\!\xi \\
&+&c r^{\alpha}\left[  W^{\mu}_{-\alpha+\frac{1+\alpha}{1+i_g},i_g+1}(x,R)+ W^{[\psi_1]}_{-\alpha+\frac{1+\alpha}{1+i_g},i_g+1}(x,R)+ W^{[\psi_2]}_{-\alpha+\frac{1+\alpha}{1+i_g},i_g+1}(x,R)\right] \\
&+&c r^{\alpha} \left\lbrace \fint_{B_r(x)}|Du|d\xi+\int_{0}^{R}[\omega(\rho)]^{\frac{1}{1+s_g}}G^{-1}\left[\fint_{B_{\rho}(x)}[G(|D\psi_1|)+G(|\psi_1|)]d\xi \right]\frac{d\rho}{\rho^{1+\alpha}} \right\rbrace.
\end{eqnarray*}
If $x$ is a Lebesgue's point of $Du$, then let $m\rightarrow\infty$, we derive
\begin{eqnarray*}
|Du(x)-S|&=&\lim_{m\rightarrow\infty}k_{m+1}  \\
&\leqslant &c \fint_{B_r(x)}(\vert Du-(Du)_{B_r(x)}\vert+\vert Du-S\vert)\operatorname{d}\!\xi \\
&+&c r^{\alpha}\left[  W^{\mu}_{-\alpha+\frac{1+\alpha}{1+i_g},i_g+1}(x,R)+ W^{[\psi_1]}_{-\alpha+\frac{1+\alpha}{1+i_g},i_g+1}(x,R)+ W^{[\psi_2]}_{-\alpha+\frac{1+\alpha}{1+i_g},i_g+1}(x,R)\right] \\
&+&c r^{\alpha} \left\lbrace \fint_{B_r(x)}|Du|d\xi+\int_{0}^{R}[\omega(\rho)]^{\frac{1}{1+s_g}}G^{-1}\left[\fint_{B_{\rho}(x)}[G(|D\psi_1|)+G(|\psi_1|)]d\xi \right]\frac{d\rho}{\rho^{1+\alpha}} \right\rbrace.
\end{eqnarray*}
If $y$ is a Lebesgue's point of $Du$, a parallel outcome ensues. Subsequently, coupling with the two previous estimates enables us to deduce
\begin{eqnarray*}
&& |Du(x)-Du(y)|\\
 &\leqslant &c \fint_{B_r(x)}(\vert Du-(Du)_{B_r(x)}\vert+\vert Du-S\vert)\operatorname{d}\!\xi \\
&+&c r^{\alpha}\left[  W^{\mu}_{-\alpha+\frac{1+\alpha}{1+i_g},i_g+1}(x,R)+ W^{[\psi_1]}_{-\alpha+\frac{1+\alpha}{1+i_g},i_g+1}(x,R)+ W^{[\psi_2]}_{-\alpha+\frac{1+\alpha}{1+i_g},i_g+1}(x,R)\right] \\
&+&c r^{\alpha} \left\lbrace \fint_{B_r(x)}|Du|d\xi+\int_{0}^{R}[\omega(\rho)]^{\frac{1}{1+s_g}}G^{-1}\left[\fint_{B_{\rho}(x)}[G(|D\psi_1|)+G(|\psi_1|)]d\xi \right]\frac{d\rho}{\rho^{1+\alpha}} \right\rbrace \\
&+&c \fint_{B_r(y)}(\vert Du-(Du)_{B_r(y)}\vert+\vert Du-S\vert)\operatorname{d}\!\xi \\
&+&c r^{\alpha}\left[  W^{\mu}_{-\alpha+\frac{1+\alpha}{1+i_g},i_g+1}(y,R)+ W^{[\psi_1]}_{-\alpha+\frac{1+\alpha}{1+i_g},i_g+1}(y,R)+ W^{[\psi_2]}_{-\alpha+\frac{1+\alpha}{1+i_g},i_g+1}(y,R)\right] \\
&+&c r^{\alpha} \left\lbrace \fint_{B_r(y)}|Du|d\xi+\int_{0}^{R}[\omega(\rho)]^{\frac{1}{1+s_g}}G^{-1}\left[\fint_{B_{\rho}(y)}[G(|D\psi_1|)+G(|\psi_1|)]d\xi \right]\frac{d\rho}{\rho^{1+\alpha}} \right\rbrace.
\end{eqnarray*}
We now choose
\begin{equation*}
S:=(Du)_{B_{3r}(x)}, \ \ \ r:=\dfrac{|x-y|}{2},
\end{equation*}
it's easy to see that  $ B_{r}(y)\subseteq B_{3r}(x) $ and therefore
\begin{eqnarray*}
&&\fint_{B_r(x)}(\vert Du-(Du)_{B_r(x)}\vert+\vert Du-S\vert)\operatorname{d}\!\xi+
\fint_{B_{r}(y)}(\vert Du-(Du)_{B_{r}(y)}\vert+\vert Du-S\vert)\operatorname{d}\!\xi  \\
&\leqslant &c(n)\fint_{B_{3r}(x)}\vert Du-(Du)_{B_{3r}(x)}\vert\operatorname{d}\!\xi.
\end{eqnarray*}
Now notice that $ x,y\in B_\frac{R}{4}(x_0)$, so $ |x-y|\leqslant \frac{R}{2} $ and then $B_{3r}(x)\subseteq B_{\frac{3R}{4}}(x)\subseteq B_R(x_0)$.
Therefore apply \eqref{1.122} to obtain
\begin{eqnarray*}
&&\fint_{B_{3r}(x)}\vert Du-(Du)_{B_{3r}(x)}\vert\operatorname{d}\!\xi \\
&\leqslant&  c r^{\alpha}M^\#_{\alpha,\frac{3R}{4}}(Du)(x)  \\
&\leqslant&c \left( \frac{r}{R}\right) ^{\alpha}\fint_{B_{\frac{3R}{4}}(x)}\vert Du\vert\operatorname{d}\!\xi \\
&+&c r^{\alpha}\left\lbrace \left[ M_{1-\alpha i_g,\frac{3R}{4}}(\mu)(x)\right] ^{\frac{1}{i_g}}+ \left[ \overline{M}_{1-\alpha i_g,\frac{3R}{4}}(\psi_1)(x)\right] ^{\frac{1}{i_g}}+ \left[ \overline{M}_{1-\alpha i_g,\frac{3R}{4}}(\psi_2)(x)\right] ^{\frac{1}{i_g}}\right\rbrace \\
&+&cr^{\alpha} \left\lbrace W_{\frac{1}{i_g+1},i_g+1}^{\mu}(x,2R)+W_{ \frac{1}{i_g+1},i_g+1}^{[\psi_1]}(x,2R)+W_{ \frac{1}{i_g+1},i_g+1}^{[\psi_2]}(x,2R)\right\rbrace  \\
&+&cr^{\alpha} \int_{0}^{2R}[\omega(\rho)]^{\frac{1}{1+s_g}}G^{-1}\left[\fint_{B_{\rho}(x)}[G(|D\psi_1|)+G(|\psi_1|)]d\xi \right]\frac{d\rho}{\rho^{1+\alpha}}.
\end{eqnarray*}
Moreover, because of  \eqref{mwanguji}, we have
\begin{eqnarray*}
\nonumber &&\fint_{B_r(x)}|Du|d\xi \\
&\leqslant& c\left[ \fint_{B_{\frac{3R}{4}}(x)}|Du|d\xi+W_{\frac{1}{i_g+1},i_g+1}^{\mu}(x,\frac{3R}{2})+W_{ \frac{1}{i_g+1},i_g+1}^{[\psi_1]}(x,\frac{3R}{2})+W_{ \frac{1}{i_g+1},i_g+1}^{[\psi_2]}(x,\frac{3R}{2})\right]  \\ \nonumber
&+&c\int_{0}^{\frac{3R}{2}}[\omega(\rho)]^{\frac{1}{1+s_g}}G^{-1}\left[\fint_{B_{\rho}(x)}[G(|D\psi_1|)+G(|\psi_1|)]d\xi \right]\frac{d\rho}{\rho},
\end{eqnarray*}
\begin{eqnarray*}
\nonumber &&\fint_{B_r(y)}|Du|d\xi\\
 &\leqslant& c\left[ \fint_{B_{\frac{3R}{4}}(y)}|Du|d\xi+W_{\frac{1}{i_g+1},i_g+1}^{\mu}(y,\frac{3R}{2})+W_{ \frac{1}{i_g+1},i_g+1}^{[\psi_1]}(y,\frac{3R}{2})+W_{ \frac{1}{i_g+1},i_g+1}^{[\psi_2]}(y,\frac{3R}{2})\right]  \\ \nonumber
&+&c\int_{0}^{\frac{3R}{2}}[\omega(\rho)]^{\frac{1}{1+s_g}}G^{-1}\left[\fint_{B_{\rho}(y)}[G(|D\psi_1|)+G(|\psi_1|)]d\xi \right]\frac{d\rho}{\rho}.
\end{eqnarray*}
Due to the preceding estimates, we derive
\begin{eqnarray*}
&&|Du(x)-Du(y)|\leqslant c\left( \frac{r}{R}\right) ^{\alpha}\fint_{B_R(x_0)}\vert Du\vert\operatorname{d}\!\xi
\\
&+&c r^{\alpha}\left\lbrace \left[ M_{1-\alpha i_g,\frac{3R}{4}}(\mu)(x)\right] ^{\frac{1}{i_g}}+ \left[ \overline{M}_{1-\alpha i_g,\frac{3R}{4}}(\psi_1)(x)\right] ^{\frac{1}{i_g}}+ \left[ \overline{M}_{1-\alpha i_g,\frac{3R}{4}}(\psi_2)(x)\right] ^{\frac{1}{i_g}}\right\rbrace \\
&+&c r^{\alpha}\left[  W^{\mu}_{-\alpha+\frac{1+\alpha}{1+i_g},i_g+1}(x,2R)+ W^{[\psi_1]}_{-\alpha+\frac{1+\alpha}{1+i_g},i_g+1}(x,2R)+ W^{[\psi_2]}_{-\alpha+\frac{1+\alpha}{1+i_g},i_g+1}(x,2R)\right] \\
&+&c r^{\alpha}\left[  W^{\mu}_{-\alpha+\frac{1+\alpha}{1+i_g},i_g+1}(y,2R)+ W^{[\psi_1]}_{-\alpha+\frac{1+\alpha}{1+i_g},i_g+1}(y,2R)+ W^{[\psi_2]}_{-\alpha+\frac{1+\alpha}{1+i_g},i_g+1}(y,2R)\right] \\
&+&c r^{\alpha}\left[ \int_{0}^{2R}[\omega(\rho)]^{\frac{1}{1+s_g}}G^{-1}\left[\fint_{B_{\rho}(x)}[G(|D\psi_1|)+G(|\psi_1|)]d\xi  \right]\frac{d\rho}{\rho^{1+\alpha}} \right] \\
&+&c r^{\alpha}\left[  \int_{0}^{2R}[\omega(\rho)]^{\frac{1}{1+s_g}}G^{-1}\left[\fint_{B_{\rho}(y)}[G(|D\psi_1|)+G(|\psi_1|)]d\xi  \right]\frac{d\rho}{\rho^{1+\alpha}}\right],
\end{eqnarray*}
where we used the fact that $W_{\frac{1}{i_g+1},i_g+1}^{\mu}(x,2R) \leqslant c W^{\mu}_{-\alpha+\frac{1+\alpha}{1+i_g},i_g+1}(x,2R)$ and for other case there are similar inequalities.

For every $ \varepsilon>0 $, we know that there exists $ 0<r\leqslant R $ such that
\begin{equation*}
M_{1-\alpha i_g,\frac{3R}{4}}(\mu)(x)\leqslant |B_1|^{-1}\frac{|\mu|(B_{\frac{3r}{4}}(x))}{\left(\frac{3r}{4} \right) ^{n-1+\alpha i_g}}+\varepsilon.
\end{equation*}
According to   the definition of Wolff potential,    we obtain
\begin{eqnarray*}
\frac{|\mu|(B_{\frac{3r}{4}}(x))}{\left(\frac{3r}{4} \right) ^{n-1+\alpha i_g}}&=&
\left[\left(\frac{|\mu|(B_{\frac{3r}{4}}(x))}{\left(\frac{3r}{4} \right) ^{n-1+\alpha i_g}} \right) ^{\frac{1}{i_g}} \frac{1}{-\log(3/4)}\int_{3r/4}^r\frac{\operatorname{d}\!\rho}{\rho}\right] ^{i_g}  \\
&\leqslant& C\left[ \int_{3r/4}^r\left(\frac{|\mu|(B_{\rho}(x))}{{\rho }^{n-1+\alpha i_g}} \right) ^{\frac{1}{i_g}} \frac{\operatorname{d}\!\rho}{\rho}\right] ^{i_g}  \\
&\leqslant& C\left[W^{\mu}_{-\alpha+\frac{1+\alpha}{1+i_g},i_g+1}(x,R) \right] ^{i_g}.
\end{eqnarray*}
Likewise,
\begin{equation*}
\frac{D\Psi_1(B_{\frac{3r}{4}}(x))}{\left(\frac{3r}{4} \right) ^{n-1+\alpha i_g}} \leqslant C\left[W^{[\psi_1]}_{-\alpha+\frac{1+\alpha}{1+i_g},i_g+1}(x,R) \right] ^{i_g},
\end{equation*}
\begin{equation*}
\frac{D\Psi_2(B_{\frac{3r}{4}}(x))}{\left(\frac{3r}{4} \right) ^{n-1+\alpha i_g}} \leqslant C\left[W^{[\psi_2]}_{-\alpha+\frac{1+\alpha}{1+i_g},i_g+1}(x,R) \right] ^{i_g}.
\end{equation*}
Finally, we consider the definition of  $r$ to derive
\begin{eqnarray*}
&&\vert Du(x)-Du(y) \vert \\
&\leq &c\fint_{B_R(x_0)}\vert Du\vert\operatorname{d}\!\xi\left(\frac{|x-y|}{R} \right) ^{\alpha}  \\
&+& c \left[  W^{\mu}_{-\alpha+\frac{1+\alpha}{1+i_g},i_g+1}(x,2R)+ W^{[\psi_1]}_{-\alpha+\frac{1+\alpha}{1+i_g},i_g+1}(x,2R)+ W^{[\psi_2]}_{-\alpha+\frac{1+\alpha}{1+i_g},i_g+1}(x,2R)\right]|x-y|^{\alpha} \\
&+&c \left[  W^{\mu}_{-\alpha+\frac{1+\alpha}{1+i_g},i_g+1}(y,2R)+ W^{[\psi_1]}_{-\alpha+\frac{1+\alpha}{1+i_g},i_g+1}(y,2R)+ W^{[\psi_2]}_{-\alpha+\frac{1+\alpha}{1+i_g},i_g+1}(y,2R)\right]|x-y|^{\alpha} \\
&+&c \left[ \int_{0}^{2R}[\omega(\rho)]^{\frac{1}{1+s_g}}G^{-1}\left[\fint_{B_{\rho}(x)}[G(|D\psi_1|)+G(|\psi_1|)]d\xi  \right]\frac{d\rho}{\rho^{1+\alpha}} \right]|x-y|^{\alpha} \\
&+&c \left[  \int_{0}^{2R}[\omega(\rho)]^{\frac{1}{1+s_g}}G^{-1}\left[\fint_{B_{\rho}(y)}[G(|D\psi_1|)+G(|\psi_1|)]d\xi  \right]\frac{d\rho}{\rho^{1+\alpha}}\right]|x-y|^{\alpha}.
\end{eqnarray*}
Then we finish the proof of Theorem \ref{Th2}.
\end{proof}

\section*{Acknowledgments}The authors are supported by 
the Fundamental Research Funds for the Central Universities 
  (Grant No.  2682024CX028) and the  National Natural Science Foundation of China (Grant No.~12071229 and 12101452).

\end{document}